\documentclass[onetabnum]{siamart1116}

\usepackage{lipsum}
\usepackage{amsfonts}
\usepackage{graphicx}
\usepackage{epstopdf}
\usepackage{algorithmic}
\ifpdf
  \DeclareGraphicsExtensions{.eps,.pdf,.png,.jpg}
\else
  \DeclareGraphicsExtensions{.eps}
\fi

\numberwithin{theorem}{section}

\newcommand{\TheTitle}{Continuity of Pontryagin Extremals with Respect to Delays in Nonlinear Optimal Control}
\newcommand{\TheShortTitle}{Continuity of Pontryagin Extremals with Respect to Delays}
\newcommand{\TheAuthors}{R. Bonalli, B. H\'{e}riss\'{e}, and E. Tr\'{e}lat}

\headers{\TheShortTitle}{\TheAuthors}

\title{{\TheTitle}
}

\author{
  Riccardo Bonalli\thanks{Sorbonne Universit\'e, Universit\'e Paris-Diderot SPC, CNRS, Inria, Laboratoire Jacques-Louis Lions, \'equipe CAGE, F-75005 Paris, France and ONERA, DTIS, Universit\'e Paris Saclay, F-91123 Palaiseau, France (\email{riccardo.bonalli@etu.upmc.fr, riccardo.bonalli@onera.fr})}
  \and
  Bruno H\'{e}riss\'{e}\thanks{ONERA, DTIS, Universit\'e Paris Saclay, F-91123 Palaiseau, France (\email{bruno.herisse@onera.fr}).}
  \and
  Emmanuel Tr\'{e}lat\thanks{Sorbonne Universit\'e, Universit\'e Paris-Diderot SPC, CNRS, Inria, Laboratoire Jacques-Louis Lions, \'equipe CAGE, F-75005 Paris, France (\email{emmanuel.trelat@sorbonne-universite.fr}).}
}

\usepackage{amsopn}

\ifpdf
\hypersetup{
pdftitle={\TheTitle},
pdfauthor={\TheAuthors}
}
\fi

\usepackage{dsfont}

\usepackage{mathtools}

\usepackage{xparse}
\usepackage{varwidth}

\newsavebox{\leftbox} \newsavebox{\rightbox}%

\NewDocumentCommand{\lrboxbrace}{s O{\{} O{\}} O{0.05\linewidth} m O{0.8\linewidth} m}{
\begin{lrbox}{\leftbox}
\IfBooleanTF{#1}
{\begin{varwidth}{#4}#5\end{varwidth}}
{\begin{minipage}{#4}#5\end{minipage}}
\end{lrbox}
\begin{lrbox}{\rightbox}
\IfBooleanTF{#1}
{\begin{varwidth}{#6}#7\end{varwidth}}
{\begin{minipage}{#6}#7\end{minipage}}
\end{lrbox}
\ensuremath{\usebox\leftbox\left#2\,\usebox\rightbox\,\right#3}
}

\usepackage{amssymb}
\usepackage{marginnote}

\begin{document}

\maketitle

\begin{abstract}
Consider a general nonlinear optimal control problem in finite dimension, with constant state and/or control delays. By the Pontryagin Maximum Principle, any optimal trajectory is the projection of a Pontryagin extremal. We establish that, under appropriate assumptions which are essentially sharp, Pontryagin extremals depend continuously on the parameters delays, for adequate topologies. The proof of the continuity of the trajectory and of the control is quite easy, however, for the adjoint vector, the proof requires a much finer analysis. The continuity property of the adjoint vector with respect to the parameter delays opens a new perspective for the numerical implementation of indirect methods, such as the shooting method.
\end{abstract}



\section{Introduction} \label{secIntro}

This paper is devoted to establishing continuity properties with respect to delays of Pontryagin extremals related to nonlinear optimal control problems with state and control constant delays. Pontryagin extremals are obtained by applying the Pontryagin Maximum Principle to an optimal control problem, thus providing first-order necessary conditions for optimality.

Historically, the Maximum Principle has been developed originally for optimal control problems without delays (see, e.g., \cite{pontryagin1987}). The paper \cite{kharatishvili1961} was first to provide a Maximum Principle for optimal control problems with constant state delays while \cite{guinn1976} obtains the same conditions by a simple substitution-like method. In \cite{kharatishvili1967} a similar result is achieved for control problems with pure control delays. In \cite{halanay1968, soliman1972}, necessary conditions are obtained for optimal control problems with multiple constant delays in state and control variables. Moreover, \cite{banks1968,asher1971} derive Maximum Principles for control systems with either time- or state-dependent delays. Finally, \cite{gollmann2009,gollmann2014} give necessary conditions for optimal control problems with delays and mixed constraints.

When delays are considered, the Maximum Principle provides extremals satisfying adjoint equations, maximality condition and transversality conditions which depend directly on the value of the delay. Therefore, it seems legitimate to wonder how these extremals depend on the parameter delay. In the present paper, we provide sufficient conditions ensuring continuity (for suitable topologies) of Pontryagin extremals with respect to delays. Continuity properties have useful numerical applications in solving optimal control problems with delays by shooting methods (as we describe in Section \ref{Sect_IndirectMethods}). \textcolor{black}{The main result presented in this paper is roughly the following:} \\

\makebox[\textwidth][l]{\textit{\parbox[]{78ex}{\textcolor{black}{\small Consider an optimal control problem with delays. Under the main assumption that optimal controls are either time continuous or purely bang-bang, Pontryagin extremals are strongly continuous with respect to delays, for appropriate topologies.}}}}

\newpage

\noindent \textcolor{black}{The assumption that optimal controls are bang-bang is sharp, in the sense that, whenever singular arcs arise, continuity of Pontryagin extremals with respect to delays may fail to be satisfied in strong topology, although it is always satisfied in a weak sense (a counterexample is provided in Section \ref{counterEx}).}

In the literature, to our knowledge, it seems that this topic has been little addressed. The main works addressing the regularity of extremals with respect to time lag variations develop sensibility analysis-type arguments (see, e.g., \cite{rihan2003,marchuk2013,marchuk2013mathematical}). Our approach developed in this paper does not require any differentiability properties of the extremals. More precisely, the main continuity result (see Theorem \ref{Ch5_TheoMain}) is achieved by analyzing the geometric deformation of Pontryagin cones, i.e., the sets containing all variation vectors, under small perturbations of the delays. This geometric analysis is challenging and requires a modified conic implicit function theorem which relies on the continuous dependence of parameters (represented here by delays).

The paper is organized as follows. In Section \ref{Sect_First}, we recall the Maximum Principle formulation for optimal control problems with delays, stating then, the main theorem for the continuity of Pontryagin extremals with respect to delays. Section \ref{Sect_IndirectMethods} contains an extension of such continuity property to implement robust shooting methods to solve optimal control problems with delays. In Section \ref{Sect_Proof}, we provide the proof of our main result, which goes in three steps. Robustness of controllability properties of problems with delays are addressed first by means of an implicit function theorem in which parameters and restriction to dense subsets are considered. In a second step, the existence of solutions of optimal control problems with delays and their continuity with respect to delays are established. The last step is the more difficult and technical: we prove the continuity of the adjoint vectors with respect to delays. Finally, Section \ref{Sect_Conclus} provides conclusions and several perspectives.

\section{Continuity of Pontryagin Extremals with Respect to Delays} \label{Sect_First}

\subsection{The Maximum Principle for Optimal Control Problems with Delays} \label{Chapter5_Preliminaries}

Let $n$, $m$ be positive integers, $\Delta$ a positive real number, $U \subseteq \mathbb{R}^m$ a measurable subset and define an initial state function $\phi^1(\cdot) \in C^0([-\Delta,0],\mathbb{R}^n)$ and an initial control function $\phi^2(\cdot) \in L^{\infty}([-\Delta,0],U)$. For $\tau = (\tau^0,\tau^1,\tau^2) \in [0,\Delta]^3$ and $t_f > 0$, consider the following nonlinear control system on $\mathbb{R}^n$ with constant delays
\begin{equation} \label{Ch5_Dyn}
\begin{cases}
\dot{x}(t) = f(t,t-\tau^0,x(t),x(t-\tau^1),u(t),u(t-\tau^2)) \quad , \quad t \in [0,t_f] \medskip \\
x(t) = \phi^1(t) \ , \ u(t) = \phi^2(t) \ , \ t \in [-\Delta,0] \quad , \quad u(\cdot) \in L^{\infty}([-\Delta,t_f],U)
\end{cases}
\end{equation}
where $f(t,s,x,y,u,v)$ is continuous and (at least) $C^2$ w.r.t. its second, third and fourth variables. Control systems (\ref{Ch5_Dyn}) play an important role describing many relevant phenomena in physics, biology, engineering and economics (see, e.g., \cite{malek1987}).

Let $M_f$ be a subset of $\mathbb{R}^{n}$. Assume that $M_f$ is reachable from $\phi^1(\cdot)$ for the control system (\ref{Ch5_Dyn}), that is, for every $\tau = (\tau^0,\tau^1,\tau^2) \in [0,\Delta]^3$, there exist a final time $t_f$ and a control $u(\cdot) \in L^{\infty}([-\Delta,t_f],U)$, such that the trajectory $x(\cdot)$, solution of (\ref{Ch5_Dyn}) in $[-\Delta,t_f]$, satisfies $x(t_f) \in M_f$. Such a control is called admissible and we denote by $\mathcal{U}^{\tau}_{t_f,\mathbb{R}^m}$ the set of all admissible controls of (\ref{Ch5_Dyn}) defined in $[-\Delta,t_f]$ taking their values in $\mathbb{R}^m$, while $\mathcal{U}^{\tau}_{t_f,U}$ denotes the set of all admissible controls of (\ref{Ch5_Dyn}) defined in $[-\Delta,t_f]$ taking their values in $U$. Therefore, $\mathcal{U}^{\tau}_{\mathbb{R}^m} = \bigcup_{t_f > 0} \mathcal{U}^{\tau}_{t_f,\mathbb{R}^m}$ and $\mathcal{U}^{\tau}_{U} = \bigcup_{t_f > 0} \mathcal{U}^{\tau}_{t_f,U}$.

Given constant delays $\tau = (\tau^0,\tau^1,\tau^2) \in [0,\Delta]^3$, we consider the Optimal Control Problem with Delays (\textbf{OCP})$_{\tau}$ consisting in steering the control system (\ref{Ch5_Dyn}) to $M_f$, while minimizing the cost function
\begin{equation} \label{Ch5_Cost}
C_{\tau}(t_f,u) = \int_{0}^{t_f} f^0(t,t-\tau^0,x(t),x(t-\tau^1),u(t),u(t-\tau^2)) \; \mathrm{d}t
\end{equation}
where $f^0(t,s,x,y,u,v)$ is continuous and (at least) $C^2$ w.r.t. its second, third and fourth variables. We study either fixed or free final time problems (\textbf{OCP})$_{\tau}$.

In the context of the present work, we focus on two particular classes of problems. We speak of problems (\textbf{OCP})$_{\tau}$ with pure state delays when $f$ and $f^0$ do not depend on $v$, i.e., $\frac{\partial f}{\partial v} = \frac{\partial f^0}{\partial v} = 0$. We say that the optimal control problem (\textbf{OCP})$_{\tau}$ is affine when $f$ and $f^0$ are affine in $(u,v)$, i.e.,
\begin{equation} \label{def_affine}
\begin{split}
f(t,s,x,y,u,v) &= f_0(t,s,x,y) + u \cdot f_1(t,s,x,y) + v \cdot f_2(t,s,x,y) \\
f^0(t,s,x,y,u,v) &= f_0^0(t,s,x,y) + u \cdot f_1^0(t,s,x,y) + v \cdot f_2^0(t,s,x,y) .
\end{split}
\end{equation}

Assume that $(x_{\tau}(\cdot),u_{\tau}(\cdot))$ is an optimal solution for (\textbf{OCP})$_{\tau}$ with related optimal final time $t^{\tau}_f$ and define the Hamiltonian related to problem (\textbf{OCP})$_{\tau}$ by
$$
H(t,s,x,y,p,p^0,u,v) = \langle p , f(t,s,x,y,u,v) \rangle + p^0 f^0(t,s,x,y,u,v) .
$$
According to the Maximum Principle (see, e.g., \cite{gollmann2014,boccia2016}), there exists a nontrivial couple $(p_{\tau}(\cdot),p^0_{\tau}) \neq 0$, where $p^0_{\tau} \leqslant 0$ is constant and $p_{\tau} : [0,t^{\tau}_f] \rightarrow \mathbb{R}^n$ (adjoint vector) is absolutely continuous, such that the so-called Pontryagin extremal $(x_{\tau}(\cdot),p_{\tau}(\cdot),p^0_{\tau},u_{\tau}(\cdot))$ satisfies, almost everywhere in $[0,t^{\tau}_f]$, the adjoint equations
\begingroup
\footnotesize
\begin{eqnarray} \label{Ch5_Adjoint}
\hspace{10pt}
\begin{cases}
\displaystyle \dot{x}_{\tau}(t) =& \displaystyle \frac{\partial H}{\partial p}(t,t-\tau^0,x_{\tau}(t),x_{\tau}(t-\tau^1),p_{\tau}(t),p^0_{\tau},u_{\tau}(t),u_{\tau}(t-\tau^2)) \medskip \\
\displaystyle \dot{p}_{\tau}(t) =& \displaystyle -\frac{\partial H}{\partial x}(t,t-\tau^0,x_{\tau}(t),x_{\tau}(t-\tau^1),p_{\tau}(t),p^0_{\tau},u_{\tau}(t),u_{\tau}(t-\tau^2)) \medskip \\
&\displaystyle -\mathds{1}_{[0,t^{\tau}_f-\tau^1]}(t) \frac{\partial H}{\partial y}(t+\tau^1,t+\tau^1-\tau^0,x_{\tau}(t+\tau^1),x_{\tau}(t),p_{\tau}(t+\tau^1), \medskip \\
&\displaystyle \hspace{75pt} p^0_{\tau},u_{\tau}(t+\tau^1),u_{\tau}(t+\tau^1-\tau^2))
\end{cases}
\end{eqnarray}
\endgroup
and the following maximality condition, for every $u \in U$
\begin{equation} \label{Ch5_Max}
\footnotesize
\begin{split}
\displaystyle H(t,t-\tau^0,x_{\tau}(t),&x_{\tau}(t-\tau^1),p_{\tau}(t),p^0_{\tau},u_{\tau}(t),u_{\tau}(t-\tau^2))(t) \medskip \\
\displaystyle &+ \mathds{1}_{[0,t^{\tau}_f-\tau^2]}(t) H(t+\tau^2,t+\tau^2-\tau^0,x_{\tau}(t+\tau^2),x_{\tau}(t+\tau^2-\tau^1), \medskip \\
&\hspace{75pt} p_{\tau}(t+\tau^2),p^0_{\tau},u_{\tau}(t+\tau^2),u_{\tau}(t)) \medskip \\
\displaystyle \geqslant H(t,t-\tau^0,x_{\tau}&(t),x_{\tau}(t-\tau^1),p_{\tau}(t),p^0_{\tau},u,u_{\tau}(t-\tau^2)) \medskip \\
\displaystyle &+ \mathds{1}_{[0,t^{\tau}_f-\tau^2]}(t) H(t+\tau^2,t+\tau^2-\tau^0,x_{\tau}(t+\tau^2),x_{\tau}(t+\tau^2-\tau^1), \medskip \\
&\hspace{75pt} p_{\tau}(t+\tau^2),p^0_{\tau},u_{\tau}(t+\tau^2),u) .
\end{split}
\end{equation}
Furthermore, if $M_f$ is a submanifold of $\mathbb{R}^n$, locally around $x_{\tau}(t^{\tau}_f)$, then the adjoint vector can be chosen in order to satisfy
\begin{eqnarray} \label{Ch5_FinalCondPerp}
p_{\tau}(t^{\tau}_f) \quad \perp \quad T_{x_{\tau}(t^{\tau}_f)} M_f
\end{eqnarray}
and, moreover, if the final time $t^{\tau}_f$ is free and both $t^{\tau}_f$, $t^{\tau}_f - \tau^2$ are Lebesgue points for $u_{\tau}(\cdot)$, the extremal $(x_{\tau}(\cdot),p_{\tau}(\cdot),p^0_{\tau},u_{\tau}(\cdot))$ satisfies the following final condition
\begin{eqnarray} \label{Ch5_FinalCond}
H(t^{\tau}_f,t^{\tau}_f-\tau^0,x_{\tau}(t^{\tau}_f),x_{\tau}(t^{\tau}_f - \tau^1),p_{\tau}(t^{\tau}_f),p^0_{\tau},u_{\tau}(t^{\tau}_f),u_{\tau}(t^{\tau}_f - \tau^2)) = 0
\end{eqnarray}
(recall that any measurable function is a.e. approximately continuous, see, e.g., \cite{evans1991}). When $t_f^\tau$ or $t_f^\tau-\tau^2$ are not Lebesgue points, (\ref{Ch5_FinalCond}) can be generalized (see, e.g., \cite{boccia2016}). The extremal $(x_{\tau}(\cdot),p_{\tau}(\cdot),p^0_{\tau},u_{\tau}(\cdot))$ is said to be normal when $p^0_{\tau} \neq 0$, and in that case we set $p^0_{\tau} = -1$. Otherwise, it is said to be abnormal.

\subsection{The Main Result: Continuity Properties of Pontryagin Extremals} \label{Sect_Teo}

Our main result establishes that, under appropriate assumptions, Pontryagin extremals are continuous with respect to delays for appropriate topologies. The most challenging issue is the continuous dependence of adjoint vectors with respect to delays. To prove this fact, we establish continuity with respect to delays of Pontryagin cones related to the Maximum Principle formulation with delays.

We will treat separately the case of pure state delays. Treating control delays happens to be more complex, especially, for the existence of optimal controls (see Theorem \ref{Ch5_TheoMain}). Indeed, a standard approach to prove existence would consider usual Filippov's assumptions (as in the classical reference \cite{filippov1962}) which, in the case of control delays, must be extended. In particular, using the Guinn's reduction (see, e.g., \cite{guinn1976}), the control system with delays is equivalent to a non-delayed system with a larger number of variables depending on the value of $\tau^2$. Such extension was used in \cite{nababan1979}. However, the usual assumption on the convexity of the epigraph of the extended dynamics is not sufficient to prove Lemma 2.1 in \cite{nababan1979} (see also Section \ref{Ch6_ExistenceSect}).

Fix constant delays $\tau_0 = (\tau^0_0,\tau^1_0,\tau^2_0) \in [0,\Delta]^3$ and let $(x_{\tau_0}(\cdot),p_{\tau_0}(\cdot),p^0_{\tau_0},u_{\tau_0}(\cdot))$ be a Pontryagin extremal for (\textbf{OCP})$_{\tau_0}$ satisfying (\ref{Ch5_Adjoint})-(\ref{Ch5_FinalCond}), where $(x_{\tau_0}(\cdot),u_{\tau_0}(\cdot))$ is an optimal solution for (\textbf{OCP})$_{\tau_0}$ in $[-\Delta,t^{\tau_0}_f]$. We make the following assumptions: \\

\textbf{General assumptions:} \\

\hspace{-30pt} \lrboxbrace[\lbrace][.][15pt]{$(A)$}[345pt]
{
\begin{enumerate}
\item[$(A_1)$] \begingroup \small $U$ is compact and convex in $\mathbb{R}^m$, and $M_f$ is a compact submanifold of $\mathbb{R}^n$. \endgroup
\item[$(A_2)$] The optimal control problem with delays (\textbf{OCP})$_{\tau_0}$ has a unique solution, denoted $(x_{\tau_0}(\cdot),u_{\tau_0}(\cdot))$, defined in a neighborhood of $[-\Delta,t^{\tau_0}_f]$.
\item[$(A_3)$] The optimal trajectory $x_{\tau_0}(\cdot)$ has a unique extremal lift (up to a multiplicative scalar) defined in $[0,t^{\tau_0}_f]$, which is normal, denoted $(x_{\tau_0}(\cdot),p_{\tau_0}(\cdot),-1,u_{\tau_0}(\cdot))$, solution of the Maximum Principle.
\item[$(A_4)$] There exists a positive real number $b$ such that, for every $\tau = (\tau^0,\tau^1,\tau^2) \in [0,\Delta]^3$ and every $v(\cdot) \in \mathcal{U}^\tau_{U}$, denoting $x_{\tau,v}(\cdot)$ the related trajectory arising from dynamics (\ref{Ch5_Dyn}) with final time $t^{\tau,v}_f$, one has
$$
\forall \ t \in [-\Delta,t^{\tau,v}_f] : \ t^{\tau,v}_f + \| x_{\tau,v}(t) \| \leqslant b .
$$
\end{enumerate}
}

\vspace{10pt}

\textbf{Additional assumptions in case of pure state delays:} \\

\hspace{-30pt} \lrboxbrace[\lbrace][.][15pt]{$(B)$}[345pt]
{
\begin{enumerate}
\item[$(B_1)$] For every $\tau$, every optimal control $u_{\tau}(\cdot)$ of (\textbf{OCP})$_{\tau}$ is continuous.
\item[$(B_2)$] The sets
\begingroup
\footnotesize
$$
\left\{ \ \left(f(t,s,x,y,u),f^0(t,s,x,y,u)+\gamma\right) \ : \ u \in U \ , \ \gamma \geqslant 0 \ \right\} ,
$$
$$
\hspace{-2.5ex} \left\{ \left(f(t,t,x,x,u),f^0(t,t,x,x,u)+\gamma,\frac{\partial \tilde f}{\partial x}(t,t,x,x,u),\frac{\partial \tilde f}{\partial y}(t,t,x,x,u)\right) : u \in U , \gamma \geqslant 0 \right\}
$$
\endgroup
are convex for every $t,s \in \mathbb{R}$, $x,y \in \mathbb{R}^n$, where we denote $\tilde f = (f,f^0)$.
\end{enumerate}
}

\vspace{10pt}

\textbf{Additional assumptions in case of delays in state and control variables:} \\

\hspace{-30pt} \lrboxbrace[\lbrace][.][15pt]{$(C)$}[345pt]
{
\textcolor{black}{\begin{enumerate}
\item[$(C_1)$] The considered optimal control problems are affine, i.e., we have \eqref{def_affine}.
\item[$(C_2)$] The final time $t_f$ is fixed.
\item[$(C_3)$] At least one of the following two conditions is satisfied: \begin{itemize}
\item For every $\tau$, every optimal control $u_{\tau}(\cdot)$ of (\textbf{OCP})$_{\tau}$ is continuous.
\item The control $u_{\tau_0}(\cdot)$ takes its values at extremal points of $U$, a.e.
\end{itemize}
\end{enumerate}}
}

\begin{theorem} \label{Ch5_TheoMain}
\textcolor{white}{ }

\begin{itemize}
\item \textbf{Optimal control problems with pure state delays:}

Denote $B^+_{\varepsilon}(\tau) = B_{\varepsilon}(\tau) \cap \mathbb{R}^2_+$. Under Assumptions $(A)$ and $(B)$:
\begin{enumerate}
\item There exists $\varepsilon_0 > 0$ such that, for every couple of delays $\tau = (\tau^0,\tau^1)$ satisfying $\| \tau - \tau_0 \| < \varepsilon_0$, each problem (\textbf{OCP})$_{\tau}$ has at least one optimal solution $(x_{\tau}(\cdot),u_{\tau}(\cdot))$ in $[-\Delta,t^{\tau}_f]$, every extremal lift of which is normal. Moreover, if the final time is fixed, then $t^{\tau}_f = t^{\tau_0}_f$ for every $\tau$.
\textcolor{black}{\item The mappings $B^+_{\varepsilon_0}(\tau_0) \ni \tau \mapsto x_\tau(\cdot)\footnote{\textcolor{black}{We notice that this correspondence is not necessarily uniquely defined. Hovewer, thanks to the previous existence statement, at least one such correspondence exists in a neighborhood of $\tau_0$ and the continuity properties hold for every such correspondence given in the following.}}$ and $B^+_{\varepsilon_0}(\tau_0) \ni \tau \mapsto p_\tau(\cdot)$ are continuous at $\tau_0$ in $C^0$ topology, and $B^+_{\varepsilon_0}(\tau_0) \ni \tau \mapsto t^\tau_f$ is continuous~at~$\tau_0$.
\item In addition, the mapping $B^+_{\varepsilon_0}(\tau_0) \ni \tau \mapsto \dot{x}_{\tau}(\cdot)$ is continuous at $\tau_0$ for the $L^\infty$ weak star topology.}
\end{enumerate}

\vspace{3pt}

\item \textbf{Optimal control problems with state and control delays:}

Denote $B^+_{\varepsilon}(\tau) = B_{\varepsilon}(\tau) \cap \mathbb{R}^3_+$. Under Assumptions $(A)$ and $(C)$:
\begin{enumerate}
\item There exists $\varepsilon_0 > 0$ such that, for every triple of delays $\tau = (\tau^0,\tau^1,\tau^2)$ satisfying $\| \tau - \tau_0 \| < \varepsilon_0$, each problem (\textbf{OCP})$_{\tau}$ has at least one optimal solution $(x_{\tau}(\cdot),u_{\tau}(\cdot))$ in $[-\Delta,t_f]$, every extremal lift of which is normal. \vspace{-10pt}
\textcolor{black}{\item The mappings $B^+_{\varepsilon_0}(\tau_0) \ni \tau \mapsto x_\tau(\cdot)$ and $B^+_{\varepsilon_0}(\tau_0) \ni \tau \mapsto p_\tau(\cdot)$ are continuous at $\tau_0$ in $C^0$ topology.
\item In addition, the mapping $B^+_{\varepsilon_0}(\tau_0) \ni \tau \mapsto (u_{\tau}(\cdot),u_{\tau}(\cdot-\tau^2))$ is continuous at $\tau_0$ for the $L^2$ weak topology. Moreover, if $u_{\tau_0}(\cdot)$ takes its values at extremal points of $U$, then the mapping $B^+_{\varepsilon_0}(\tau_0) \ni \tau \mapsto (u_{\tau}(\cdot),u_{\tau}(\cdot-\tau^2))$ is continuous at $\tau_0$ in $L^{\infty}$ topology.}
\textcolor{black}{\item When $\tau^2 = 0$, the previous conclusions remain valid for problems in free final time, and then, the mapping $B^+_{\varepsilon_0}(\tau_0) \ni \tau \mapsto t^\tau_f$ is continuous at $\tau_0$.}
\end{enumerate}
\end{itemize}
\end{theorem}

Several remarks are in order.

First of all, Assumptions $(A_2)$ and $(A_3)$ on the uniqueness of the solution of (\textbf{OCP})$_{\tau_0}$ and on the uniqueness of its extremal lift are ``generic": they are actually related to the differentiability properties of the value function (see, e.g., \cite{clarke1987,aubin2009,rifford2009}). They are standard in optimization and are just made to keep a nice statement (see Theorem \ref{Ch5_TheoMain}). These assumptions can be weakened as follows. If we replace $(A_2)$ and $(A_3)$ with the assumption ``every extremal lift of every solution of (\textbf{OCP})$_{\tau_0}$ is normal", then, the conclusion provided in Theorem \ref{Ch5_TheoMain} below still holds, except that the continuity properties must be written in terms of closure points (see Section \ref{Ch6_SectMainProof} and \cite[Remark 1.11]{haberkorn2011}). Finally, requiring that the unique extremal of (\textbf{OCP})$_{\tau_0}$ is moreover normal is crucial to prove the continuity of the adjoint vectors w.r.t. delays.

\textcolor{black}{Assumptions $(B_1)$ and $(C_3)$ play a complementary role in proving continuity for the adjoint vectors (see Section \ref{Ch6_ConvAdjoint}). They are related to the strict Legendre-Clebsch condition and the uniqueness of solutions for (\ref{Ch5_Max}), and, even if these assumptions seem to be restrictive, many problems in applications satisfy them and examples can be found in \cite{gollmann2009,gollmann2014,rihan2003,bonalli2017,silva2017,silva2018,gollmann2018}. In particular, Assumption $(C_3)$ is instrumentally used also to derive strong continuity of controls from weak continuity when there are delays on the control, because of the following general fact. Let $X$, $Y$ be Banach spaces and $F : X \rightarrow Y$ be a continuous map. Suppose that $(x_k)_{k \in \mathbb{N}} \subseteq X$ is a sequence such that $x_k \rightharpoonup x$ and $F(x_k) \rightharpoonup F(\bar x)$ for some $x, \bar x \in X$. Therefore, in general, we cannot ensure that the two limits coincide, that is $x = \bar x$, except when $F$ is linear and injective. On the other hand, Assumption $(C_2)$ becomes essential to ensure continuity of Pontryagin cones when there are delays in the control variables (see Section \ref{Ch6_ConvAdjoint}).}
	
\textcolor{black}{When the problem is control-affine, we assume in (one of the two possibilities of) Assumption $(C_3)$ that the optimal control $u_{\tau_0}(\cdot)$ takes its values in the extremal set of $U$. This is the case when the optimal control is bang-bang. As said above, this property permits to turn weak continuity into strong continuity. Such an assumption is sharp: the counterexample of Section \ref{counterEx} shows that, when the optimal control of (\textbf{OCP})$_{\tau_0}$ does not take its values at extremal points of $U$, only a weak continuity of optimal controls can be ensured in general.}

\begin{remark} \label{Ch5_RemarkExample}
The conclusions of Theorem \ref{Ch5_TheoMain} are valid for optimal control problems with control-affine dynamics and quadratic cost
$$
\int_0^{t_f} \Big( K_1 \| x(t) \|^2 + K_2 \| x(t-\tau^1) \|^2 + K_3 \| u(t) \|^2 + K_4 \| u(t-\tau^2) \|^2 \Big) \; \mathrm{d}t
$$
(see Section \ref{Ch6_SectMainProof} and the proof in \cite[Theorem 1]{bonalli2017}).
\end{remark}

\subsection{Weak Continuity Versus Strong Continuity of Optimal Controls} \label{counterEx}

For control-affine problems, Theorem \ref{Ch5_TheoMain} ensures weak continuity in $L^2$ of optimal controls. Moreover, when the optimal control $u_{\tau_0}(\cdot)$ takes its values at extremal points of $U$ a.e., continuity is true in strong $L^{\infty}$ topology. In this section, we provide an example where $u_{\tau_0}(\cdot)$ does not take its values at extremal points of $U$, and continuity fails in strong topology. This example shows that our assumptions are sharp.

Adapting arguments from \cite{silva2010}, consider (\textbf{OCP})$_{\tau}$ given by
\begingroup
\small
\begin{eqnarray} \label{ex_dyn}
\begin{cases}
\displaystyle \quad \min \ \int_{0}^{t_f} 1 \; \mathrm{d}t \quad , \quad (u^1(t))^2 + (u^2(t))^2 \leqslant 1 \medskip \\
\displaystyle \dot{x}^1(t) = 1 - (x^2(t))^2 + \tau u^2(t) g(x^1(t)) \quad , \quad x^1(0) = 0 \ , \ x^1(t_f) = 1 \medskip \\
\displaystyle \dot{x}^2(t) = u^1(t) + \tau u^2(t) h(x^1(t)) \quad , \quad x^2(0) = 0 \ , \ x^2(t_f) = 0
\end{cases}
\end{eqnarray}
\endgroup
where $g$ and $h$ are smooth functions, to be chosen. In this case
$$
f(t,s,x^1,x^2,u^1,u^2) = \left( \begin{array}{c}
1 - (x^2)^2 + (t - s) u^2 g(x^1) \\
u^1 + (t - s) u^2 h(x^1)
\end{array} \right) \ .
$$
Taking $\tau_0 = 0$, under appropriate assumptions on $g$ and $h$, Theorem \ref{Ch5_TheoMain} applies with weak convergence in $L^2$ of controls, but no strong convergence of controls arises. The proof follows closely the arguments of \cite{silva2010}, therefore, we just recall the main steps.

First of all, when $\tau = 0$, problem (\ref{ex_dyn}) has the unique solution $(u^1,u^2)(t) = 0$ with unique extremal $(x^1,x^2,p^1,p^2,p^0,u^1,u^2)(t) = (t,0,1,0,-1,0,0)$, where $t_f = 1$.

Remark that, in the case in which Theorem \ref{Ch5_TheoMain} applies, only the weak convergence in $L^2$ of optimal controls is ensured. The assumptions of Theorem \ref{Ch5_TheoMain} hold, in particular, every optimal control of (\textbf{OCP})$_{\tau}$ is continuous. Indeed, in \cite{silva2010} it is shown that, in the case $\tau \neq 0$, under the assumption that function $g$ may only vanish on a subset of zero measure, the optimal controls related to problem (\ref{ex_dyn}) are
\begin{equation} \label{ex_control}
\displaystyle u^1_{\tau}(t) = \frac{p^2_{\tau}(t)}{\sqrt{\phi(t)}} \quad , \quad u^2_{\tau}(t) = \frac{\tau \Big( p^1_{\tau}(t) g(x^1_{\tau}(t)) + p^2_{\tau}(t) h(x^1_{\tau}(t)) \Big)}{\sqrt{\phi(t)}}
\end{equation}
where $\phi(t) = (p^2_{\tau}(t))^2 + \tau^2 \Big( p^1_{\tau}(t) g(x^1_{\tau}(t)) + p^2_{\tau}(t) h(x^1_{\tau}(t)) \Big)^2$, and the adjoint coordinate $p^2_{\tau}$ may only vanish on subsets of zero measure. It follows that optimal controls (\ref{ex_control}) are continuous, and therefore, Theorem \ref{Ch5_TheoMain} applies.

Both $u^1_{\tau}$ and $u^2_{\tau}$ converge weakly to 0 as soon as $\tau$ tends to 0. However, in \cite{silva2010} it is also proved that specific choices of highly oscillating functions $g$ and $h$ provide that $u^1_{\tau}$ and $u^2_{\tau}$ cannot converge almost everywhere to 0 when $\tau$ tends to 0.

\subsection{Application to Shooting Methods} \label{Sect_IndirectMethods}

In this section, we briefly discuss how the continuity properties of Pontryagin extremals stated in Theorem \ref{Ch5_TheoMain} may be exploited to solve optimal control problems with delays via shooting methods.

It is known that solving generic (\textbf{OCP})$_{\tau}$ via shooting methods may be difficult. \textcolor{black}{A first difficulty is to express optimal controls as functions of $(x_{\tau}(\cdot),p_{\tau}(\cdot))$ (by the maximality condition (\ref{Ch5_Max})). We limit ourselves to highlight that this becomes possible for a large number of applications, as specified in Section \ref{Sect_Teo} where we refer to various examples. Now, assuming that one is able to provide optimal controls as functions of $(x_{\tau}(\cdot),p_{\tau}(\cdot))$, each iteration of a shooting method consists in solving the coupled dynamics (\ref{Ch5_Adjoint}), where a value of $p_{\tau}(0)$ is provided. This means that one has to solve a differential-difference boundary value problem where both forward and backward terms of time appear within mixed type differential equations.} It follows that, in order to initialize successfully a shooting method for (\ref{Ch5_Adjoint}), a guess of the initial value of the adjoint vector $p_{\tau}(0)$ is not sufficient, but rather, a good numerical guess of the whole function $p_{\tau}(\cdot)$ must be provided to make the procedure converge. This represents an additional difficulty with respect to the usual shooting method and a global discretization of the system of differential equations (\ref{Ch5_Adjoint}) must be performed. The result stated in Theorem \ref{Ch5_TheoMain} suggests that one may solve (\textbf{OCP})$_{\tau}$ numerically via shooting methods iteratively, starting from the solution of its non-delayed version (\textbf{OCP}) = (\textbf{OCP})$_{\tau = 0}$, and this by means of homotopy procedures.

The basic idea of homotopy methods is to solve a difficult problem step by step, starting from a simpler problem, by parameter deformation. Theory and practice of homotopy methods are well known (see, e.g., \cite{allgower2003}). Combined with the shooting problem derived from the Maximum Principle, a homotopy method consists in deforming the problem into a simpler one (that can be easily solved) and then in solving a series of shooting problems step by step to come back to the original problem. In many situations, exploiting the non-delayed version of the Maximum Principle mixed to other techniques (such as geometric control, dynamical system theory applied to mission design, etc., we refer the reader to \cite{trelat2012} for a survey), one is able to initialize efficiently a shooting method on the optimal control problem without delays (\textbf{OCP}). Thus, it is legitimate to wonder if one may solve (\textbf{OCP})$_{\tau}$ by shooting methods starting a homotopy procedure where delays $\tau$ represent the deformation parameter and (\textbf{OCP}) is taken as the starting problem. This approach is a way to address the flaw of shooting methods applied to (\textbf{OCP})$_{\tau}$: on one hand, the global adjoint vector for (\textbf{OCP}) could be used to initialize efficiently a shooting method on (\ref{Ch5_Adjoint}) and, on the other hand, we could solve (\ref{Ch5_Adjoint}) via usual iterative methods for ODEs (for example, by using the global state solution at the previous iteration). \textcolor{black}{In the context of Theorem \ref{Ch5_TheoMain}, the previous homotopy procedure is well-posed. Indeed, assuming that the delay $\tau$ is small enough and one is able to provide optimal controls $u_{\tau}(\cdot)$ as functions of $x_{\tau}(\cdot)$ and $p_{\tau}(\cdot)$, the continuity properties stated by Theorem \ref{Ch5_TheoMain} applied at $\tau = 0$ allow straightforwardly the homotopic path to converge towards extremals related to (\textbf{OCP})$_{\tau}$ when starting from the extremal of (\textbf{OCP}).}

Implementing homotopy methods combined with the Maximum Principle with delays, as shortly described above, however requires a number of considerations which go beyond the scope of the present paper. In a forthcoming work, we will show that this approach indeed happens to be competitive with respect to other approaches (in particular, direct methods that have been classically applied in this context, see, e.g., \cite{gollmann2009,gollmann2014}) and we will illustrate it on a nontrivial and nonacademic optimal control problem with launch vehicles, a case of application where delays are the main obstacle to \textcolor{black}{efficiently embark numerical approaches (the authors already showed the numerical efficiency of this approach on a simpler problem, see \cite{bonalli2017}).}

\section{Proof of the Main Result} \label{Sect_Proof}

Without loss of generality, we assume that $\tau_0 = 0$, denoting by (\textbf{OCP})=(\textbf{OCP})$_{\tau_0}$ the problem without delays. Furthermore, we denote $t_f = t^{\tau_0}_f$, $x(\cdot) = x_{\tau_0}(\cdot)$, $p(\cdot) = p_{\tau_0}(\cdot)$ and $u(\cdot) = u_{\tau_0}(\cdot)$.

The proof of the convergence of extremals for (\textbf{OCP})$_{\tau}$ to the extremal of (\textbf{OCP}), as $\tau$ tends to 0, is organized in three steps. First, by using assumptions on the non-delayed version of the problem only, we infer the controllability of problems (\textbf{OCP})$_{\tau}$, for every $\tau$ small enough. The previous step requires some implicit function theorem involving parameters. This allows to proceed to the second part, which consists in showing the existence of solutions for (\textbf{OCP})$_{\tau}$, for $\tau$ sufficiently small, and their convergences, as $\tau$ tends to 0, to solutions of (\textbf{OCP}). In the case of control and state delays, we will see the importance of considering control-affine systems throughout this step. Finally, we address the more difficult issue of establishing the convergence, as $\tau$ tends to 0, of the adjoint vectors related to (\textbf{OCP})$_{\tau}$ to the adjoint vector of (\textbf{OCP}): a refined analysis on the convergence of Pontryagin cones is needed.

We recall in Section \ref{Ch6_SectClassicPMP} the main steps of the proof related to the Maximum Principle for problems (\textbf{OCP})$_{\tau}$ with delays. To our knowledge, the proof of this result via needle-like variations does not appear explicitly in the literature, and for the comprehension of the whole proof, we report it. In a second time, Section \ref{Ch6_SectConic} provides a useful conic version of the implicit function theorem, depending on parameters, which is a key element for the entire reasoning. Finally, in Section \ref{Ch6_SectMainProof}, we report the whole proof of the desired result, as detailed above.

\subsection{Proof of the PMP Using Needle-Like Variations} \label{Ch6_SectClassicPMP}

In this section we sketch the proof of the Maximum Principle for (\textbf{OCP})$_{\tau}$ using needle-like variations. For this, we do not use the assumptions of Theorem \ref{Ch5_TheoMain}, giving the result for a larger class of control systems with constant delays. Our reasoning is valid as well for problems with free final time, since we do not employ the well known reduction to a fixed final time problem, but rather, we modify the Pontryagin cone to keep track of the free variable $t^{\tau}_f$, by making $L^1$-variations on $t^{\tau}_f$ (as in \cite{kharatishvili1961}, for pure state delays).

\subsubsection{Preliminary Notations}
Fix a constant delay $\tau = (\tau^0,\tau^1,\tau^2) \in [0,\Delta]^3$. Consider (\textbf{OCP})$_{\tau}$ as given by formulation (\ref{Ch5_Dyn})-(\ref{Ch5_Cost}) and let $(x_{\tau}(\cdot),u_{\tau}(\cdot))$ be an optimal solution defined in $[-\Delta,t^{\tau}_f]$. For every positive final time $\bar t_f$, introduce the instantaneous cost function $x^0_{\tau}(\cdot)$ defined in $[-\Delta,\bar t_f]$ and solution of
$$
\begin{cases}
\dot{x}^0(t) = f^0(t,t-\tau^0,x_{\tau}(t),x_{\tau}(t-\tau^1),u_{\tau}(t),u_{\tau}(t-\tau^2)) \ , \quad t \in [0,\bar t_f] \medskip \\
x^0(t) = 0 \ , \ t \in [-\Delta,0]
\end{cases}
$$
such that (\ref{Ch5_Cost}) provides $C_{\tau}(\bar t_f,u_{\tau}) = x^0_{\tau}(\bar t_f)$. We define the extended state $\tilde x = (x,x^0)$ and the extended dynamics $\tilde f(t,s,\tilde x,\tilde y,u,v) = (f(t,s,x,y,u,v),f^0(t,s,x,y,u,v))$, for which, we will often denote $\tilde f(t,s,x,y,u,v) = \tilde f(t,s,\tilde x,\tilde y,u,v)$. Consider the extended dynamical problem in $\mathbb{R}^{n+1}$
\begin{equation} \label{Ch6_DynDelayAug}
\begin{cases}
\dot{\tilde{x}}(t) = \tilde f(t,t-\tau^0,\tilde x(t),\tilde x(t-\tau^1),u(t),u(t-\tau^2)) \ , \quad t \in [0,\bar t_f] \medskip \\
\tilde x|_{[-\Delta,0]}(t) = (\phi^1(t),0) \quad , \quad \tilde x(\bar t_f) \in M_f \times \mathbb{R} \medskip \\
u(\cdot) \in L^{\infty}([-\Delta,\bar t_f],U) \quad , \quad  u|_{[-\Delta,0]}(t) = \phi^2(t) \ .
\end{cases}
\end{equation}

As provided in Section \ref{Chapter5_Preliminaries}, the set of all admissible controls of (\ref{Ch6_DynDelayAug}) in $[-\Delta,\bar t_f]$ taking their values in $\mathbb{R}^m$ is denoted by $\tilde{\mathcal{U}}^{\tau}_{\bar t_f,\mathbb{R}^m}$, while $\tilde{\mathcal{U}}^{\tau}_{\bar t_f,U}$ denotes the set of all admissible controls of (\ref{Ch6_DynDelayAug}) in $[-\Delta,\bar t_f]$ taking their values in $U$. From this
$$
\tilde{\mathcal{U}}^{\tau}_{\mathbb{R}^m} = \bigcup_{\bar t_f > 0} \tilde{\mathcal{U}}^{\tau}_{\bar t_f,\mathbb{R}^m} \quad , \quad \tilde{\mathcal{U}}^{\tau}_{U} = \bigcup_{\bar t_f > 0} \tilde{\mathcal{U}}^{\tau}_{\bar t_f,U} .
$$

The extended end-point mapping is defined as
$$
\tilde E_{\tau,\bar t_f} : \tilde{\mathcal{U}}^{\tau}_{\bar t_f,\mathbb{R}^m} \rightarrow \mathbb{R}^{n+1} : u \mapsto \tilde{x}(\bar t_f)
$$
where $\tilde x(\cdot)$ is the unique solution of problem (\ref{Ch6_DynDelayAug}), related to control $u(\cdot) \in \tilde{\mathcal{U}}^{\tau}_{\bar t_f,\mathbb{R}^m}$. As standard facts (see, e.g., \cite{bonnard2003}), the set $\tilde{\mathcal{U}}^{\tau}_{\bar t_f,\mathbb{R}^m}$, endowed with the standard topology of $L^{\infty}([-\Delta,\bar t_f],\mathbb{R}^m)$, is open and the end-point mapping is smooth on $\tilde{\mathcal{U}}^{\tau}_{\bar t_f,\mathbb{R}^m}$.

For every $t \geqslant 0$, define the extended accessible set $\tilde{\mathcal{A}}_{\tau,U}(t)$ as the image of the extended end-point mapping $\tilde E_{\tau,t}$ restricted to $\tilde{\mathcal{U}}^{\tau}_{t,U}$, where $\tilde{\mathcal{A}}_{\tau,U}(0) = \{ (\phi^1(0),0) \}$. The next fact is at the basis of the proof of the Maximum Principle (see, e.g., \cite{agrachev2013}).
\begin{lemma} \label{Ch6_LemmaInterior}
For every optimal solution $(x_{\tau}(\cdot),u_{\tau}(\cdot))$ of (\textbf{OCP})$_{\tau}$ defined in the interval $[-\Delta,t^{\tau}_f]$, the point $\tilde x_{\tau}(t^{\tau}_f)$ belongs to the boundary of the set $\tilde{\mathcal{A}}_{\tau,U}(t^{\tau}_f)$.
\end{lemma}

\subsubsection{Needle-Like Variations and Pontryagin Cones} \label{Ch6_SectNeedle}

In what follows we consider (\textbf{OCP})$_{\tau}$ whit free final time, remarking that all results can be adapted for problems with fixed final time (see, e.g., \cite{agrachev2013}). Moreover, we suppose that the optimal final time $t^{\tau}_f$ is a Lebesgue point for the optimal control $u_{\tau}(\cdot)$ of (\textbf{OCP})$_{\tau}$ and of $u_{\tau}(\cdot-\tau_2)$. Otherwise, we can extend all the conclusions that follow by using closure points near $t^{\tau}_f$ \textcolor{black}{(as in \cite[pages 310-314]{lee1967} or \cite[pages 133-134]{gamkrelidze2013}).} \\

For delays $\tau = (\tau^0,\tau^1,\tau^2) \in [0,\Delta]^3$, let $(x_{\tau}(\cdot),u_{\tau}(\cdot))$ be a solution of (\textbf{OCP})$_{\tau}$ and, without loss of generality, extend $u_{\tau}(\cdot)$ by some constant vector of $U$ in $[t^{\tau}_f,t^{\tau}_f+\tau^2]$. Let $j \geqslant 1$ be an integer and consider $0 < t_1 < \dots < t_j < t^{\tau}_f$ Lebesgue points respectively of $u_{\tau}(\cdot)$, $u_{\tau}(\cdot-\tau^2)$ and of $u_{\tau}(\cdot+\tau^2)$. Choosing $j$ arbitrary values $u_i \in U$, for every $\eta_i > 0$ such that $-\Delta \leqslant t_i - \eta_i$, the needle-like variation $\pi = \{ t_1,\dots,t_j,\eta_1,\dots,\eta_j,u_1,\dots,u_j \}$ of control $u_{\tau}(\cdot)$ is defined by the modified control
$$
u^{\pi}_{\tau}(t) = \begin{cases} u_i & t \in (t_i - \eta_i,t_i] \ , \\ u_{\tau}(t) & \textnormal{otherwise} \ . \end{cases}
$$
Control $u^{\pi}_{\tau}(\cdot)$ takes its values in $U$ and, by continuity with respect to initial data, whenever $\| (\eta_1,\dots,\eta_j) \| \rightarrow 0$, the trajectory $\tilde x^{\pi}_{\tau}(\cdot)$, solution of the dynamics of (\ref{Ch6_DynDelayAug}) related to control $u^{\pi}_{\tau}(\cdot)$, converges uniformly to $\tilde x_{\tau}(\cdot) = (x^0_{\tau}(\cdot),x_{\tau}(\cdot))$. For every value $z \in U$ and appropriate Lebesgue point $s\in (0,t^{\tau}_f)$, we define the vectors
\begin{eqnarray} \label{Ch6_OmegaMinus}
\omega^-_z(s) &=& \tilde f(s,s-\tau^0,x_{\tau}(s),x_{\tau}(s-\tau^1),z,u_{\tau}(s-\tau^2)) \\
&-& \tilde f(s,s-\tau^0,x_{\tau}(s),x_{\tau}(s-\tau^1),u_{\tau}(s),u_{\tau}(s-\tau^2)) \nonumber
\end{eqnarray}
\begin{eqnarray} \label{Ch6_OmegaPlus}
\hspace{10pt} \omega^+_z(s) &=& \tilde f(s+\tau^2,s+\tau^2-\tau^0,x_{\tau}(s+\tau^2),x_{\tau}(s+\tau^2-\tau^1),u_{\tau}(s+\tau^2),z) \\
&-& \tilde f(s+\tau^2,s+\tau^2-\tau^0,x_{\tau}(s+\tau^2),x_{\tau}(s+\tau^2-\tau^1),u_{\tau}(s+\tau^2),u_{\tau}(s)) \nonumber
\end{eqnarray}
and, given $\xi \in \mathbb{R}^{n+1}$, we denote by $\tilde v^{\tau}_{s,\xi}(\cdot)$ the solution of the following linear system
\begin{eqnarray} \label{Ch6_DynVariation}
\begin{cases}
\dot{\psi}(t) = \displaystyle \frac{\partial \tilde f}{\partial x}(t,t-\tau^0,x_{\tau}(t),x_{\tau}(t-\tau^1),u_{\tau}(t),u_{\tau}(t-\tau^2)) \psi(t) \medskip \\
\displaystyle \hspace{25pt} + \frac{\partial \tilde f}{\partial y}(t,t-\tau^0,x_{\tau}(t),x_{\tau}(t-\tau^1),u_{\tau}(t),u_{\tau}(t-\tau^2)) \psi(t-\tau^1) \medskip \\
\psi(s) = \xi \quad , \quad \psi(t) = 0 \ , \ t \in (s-\tau^1,s) \ .
\end{cases}
\end{eqnarray}
Functions $\tilde v^{\tau}_{s,\xi} : \mathbb{R} \rightarrow \mathbb{R}^{n+1}$ are usually called variations vectors. In what follows, we denote $\tilde w^{\tau}_{s,z}(t) = \tilde v^{\tau}_{s,\omega^-_z(s)}(t) + \tilde v^{\tau}_{s+\tau^2,\omega^+_z(s)}(t)$.

\begin{definition} \label{Ch6_DefCone}
\textcolor{black}{For every $t \in (0,t^{\tau}_f]$, the Pontryagin cone $\tilde K^{\tau}(t) \subseteq \mathbb{R}^{n+1}$ at $\tilde x_{\tau}(t)$ for the extended system is defined as the smallest closed convex cone containing vectors $\tilde w^{\tau}_{s,z}(t)$ where $z \in U$ and $0 < s < t$ is a Lebesgue point of $u_{\tau}(\cdot)$, $u_{\tau}(\cdot-\tau^2)$ and of $u_{\tau}(\cdot+\tau^2)$. The augmented Pontryagin cone $\tilde K^{\tau}_1(t) \subseteq \mathbb{R}^{n+1}$ at $\tilde x_{\tau}(t)$ for the extended system is defined as the smallest closed convex cone containing $\tilde f(t,t-\tau^0,x_{\tau}(t),x_{\tau}(t-\tau^1),u_{\tau}(t),u_{\tau}(t-\tau^2))$, $-\tilde f(t,t-\tau^0,x_{\tau}(t),x_{\tau}(t-\tau^1),u_{\tau}(t),u_{\tau}(t-\tau^2))$ and vectors $\tilde w^{\tau}_{s,z}(t)$ where $z \in U$ and $0 < s < t$ is a Lebesgue point of $u_{\tau}(\cdot)$, $u_{\tau}(\cdot-\tau^2)$ and of $u_{\tau}(\cdot+\tau^2)$. The Pontryagin cone $K^{\tau}(t) \subseteq \mathbb{R}^{n}$ and the augmented Pontryagin cone $K^{\tau}_1(t) \subseteq \mathbb{R}^{n}$ at $x_{\tau}(t)$ for the non-extended system are defined similarly, considering dynamics $f$ instead of the extended dynamics $\tilde f$. Obviously, the cones $K^{\tau}(t)$, $K^{\tau}_1(t)$ are the projections onto $\mathbb{R}^n$ of $\tilde K^{\tau}(t)$, $\tilde K^{\tau}_1(t)$, respectively.}
\end{definition}

\begin{remark} \label{Ch6_RemarkVariations}
In the case of optimal control problems without delays, that is (\textbf{OCP})$_{\tau=0}$, the definition of Pontryagin cones is slightly different from the one obtained from Definition \ref{Ch6_DefCone} with the substitution $\tau = 0$. Indeed, considering $\tilde K^0(t)$, vectors $\tilde w^{\tau}_{s,z}(t)$ are rather replaced by single variations $\tilde v^{0}_{s,\omega_z(s)}(t)$ for which
$$
\omega_z(s) = \tilde f(s,s,x_0(s),x_0(s),z,z) - \tilde f(s,s,x_0(s),x_0(s),u_0(s),u_0(s))
$$
where $(x_0(\cdot),u_0(\cdot))$ is an optimal solution for (\textbf{OCP})$_{\tau=0}$ (see, e.g., \cite{pontryagin1987}).
\end{remark}

\begin{lemma} [Needle-Like Variation Formula with Delays]\label{Ch6_LemmaNeedleLike}
Let $(\delta,\eta_1,\dots,\eta_j) \in \mathbb{R} \times \mathbb{R}^j_+$ small enough. For $t_j< t \leqslant t^{\tau}_f$ Lebesgue point of $u_{\tau}(\cdot)$, $u_{\tau}(\cdot-\tau^2)$, we have
\begin{multline*}
\tilde x^{\pi}_{\tau}(t + \delta) = \tilde x_{\tau}(t) + \delta \tilde f(t,t-\tau^0,x_{\tau}(t),x_{\tau}(t-\tau^1),u_{\tau}(t),u_{\tau}(t-\tau^2)) \medskip \\
+ \sum_{i=1}^{j} \eta_i \Big( \tilde v^{\tau}_{t_i,\omega^-_{u_i}(t_i)}(t) + \tilde v^{\tau}_{t_i+\tau^2,\omega^+_{u_i}(t_i)}(t) \Big) + o \Big( \delta + \sum_{i=1}^{j} \eta_i \Big) .
\end{multline*}
\end{lemma}

The proof of this lemma is technical (but not difficult). It is done in Appendix~\ref{appendixProof}.

\subsubsection{Proof of The Maximum Principle} \label{Ch6_SectProofPMPDelay}

One way to prove the Maximum Principle is by contradiction via the following classical result (see, e.g., \cite{agrachev2013,dmitruk1980,bonnard2006}).

\begin{lemma}[Conic Implicit Function Theorem] \label{Ch6_IFT}
\textcolor{black}{Let $C \subseteq \mathbb{R}^m$ be convex with non empty interior, of vertex 0, and $F : C \rightarrow \mathbb{R}^n$ be Lipschitz such that $F(0) = 0$ and G\^{a}teaux differentiable at 0 along admissible directions of $C$, i.e., there exists a linear mapping $dF(0) : \mathbb{R}^m \rightarrow \mathbb{R}^n$ such that, for every $x \in C$
\begin{equation} \label{Ch6_derConvex}
\frac{F(\alpha x)}{\alpha} \underset{\begingroup \tiny \begin{array}{c}
\hspace{5pt} \alpha \rightarrow 0^+
\end{array} \endgroup}{\xrightarrow{\hspace*{20pt}}} dF(0) x \ .
\end{equation}
If $dF(0)\cdot \textnormal{Cone}(C) = \mathbb{R}^n$, where $\textnormal{Cone}(C)$ stands for the (convex) cone generated by elements of $C$, then $0 \in \textnormal{Int} \; F(\mathcal{V} \cap C)$, for every neighborhood $\mathcal{V}$ of 0 in $\mathbb{R}^m$.}
\end{lemma}

Consider any integer $j \geqslant 1$ and a positive real number $\varepsilon_j > 0$. Define
$$
G^{\tau}_j : B_{\varepsilon_j}(0) \cap \mathbb{R} \times  \mathbb{R}^j_+ \rightarrow \mathbb{R}^{n+1} : (\delta,\eta_1,\dots,\eta_j)\mapsto \tilde x^{\pi}_{\tau}(t^{\tau}_f + \delta) - \tilde x_{\tau}(t^{\tau}_f)
$$
where $\pi$ is any variation of control $u_{\tau}(\cdot)$ and $\varepsilon_j$ is small enough such that $G^{\tau}_j$ is well-defined (see Section \ref{Ch6_SectNeedle}). The following statements hold:
\begin{itemize}
\item $G^{\tau}_j(0) = 0$ and $G^{\tau}_j$ is Lipschitz continuous.
\item $G^{\tau}_j$ is G\^{a}teaux differentiable at 0 along admissible directions of the convex set $B_{\varepsilon_j}(0) \cap \mathbb{R} \times  \mathbb{R}^j_+$ (in the sense of (\ref{Ch6_derConvex})) thanks to Lemma \ref{Ch6_LemmaNeedleLike}.
\end{itemize}
The Lipschitz behavior of $G^{\tau}_j$ is proved by a recursive use of needle-like variations at $t_i - \eta_i$, $1 \leqslant i \leqslant j$ (for $\eta_i$ small enough), Lebesgue points of $u_{\tau}(\cdot)$, by making a recursive use of Lemma \ref{Ch6_LemmaNeedleLike}. Remark that, since $t_i - \eta_i$ are Lebesgue points of $u_{\tau}(\cdot)$ only for almost every $\eta_i$, the recursive use of Lemma \ref{Ch6_LemmaNeedleLike} can be done only almost everywhere. The conclusion follows from the continuity of $G^{\tau}_j$ and density arguments.

The Maximum Principle is established as follows. Suppose, by contradiction, that the cone $\tilde K^{\tau}_1(t^{\tau}_f)$ coincides with $\mathbb{R}^{n+1}$. Then, by definition, there would exist an integer $j \geqslant 1$, a variation $\pi$ of $u_{\tau}(\cdot)$ and a positive real number $\varepsilon_j > 0$ such that
$$
d G^{\tau}_j(0)\cdot(\mathbb{R} \times \mathbb{R}^j_+) = \tilde K^{\tau}_1(t^{\tau}_f) = \mathbb{R}^{n+1} .
$$
In this case, Lemma \ref{Ch6_IFT} would imply that the point $\tilde x_{\tau}(t^{\tau}_f)$ belongs to the interior of the accessible set $\tilde{\mathcal{A}}_{\tau,U}(t^{\tau}_f)$, which contradicts Lemma \ref{Ch6_LemmaInterior}. Therefore:

\begin{lemma} \label{Ch6_LemmaMultiplier}
There exists $\tilde \psi_{\tau} \in \mathbb{R}^{n+1} \setminus \{ 0 \}$ (\textit{Lagrange multiplier}) such that
\begin{eqnarray*}
&\langle \tilde \psi_{\tau} , \tilde f(t^{\tau}_f,t^{\tau}_f-\tau^0,x_{\tau}(t^{\tau}_f),x_{\tau}(t^{\tau}_f-\tau^1),u_{\tau}(t^{\tau}_f),u_{\tau}(t^{\tau}_f-\tau^2)) \rangle = 0 , \medskip \\
&\langle \tilde \psi_{\tau} , \tilde v_{\tau} \rangle \leqslant 0 \quad , \quad \forall \ \tilde v_{\tau} \in \tilde K^{\tau}(t^{\tau}_f) .
\end{eqnarray*}
\end{lemma}

The relations provided by Lemma \ref{Ch6_LemmaMultiplier} allow to derive the necessary conditions (\ref{Ch5_Adjoint})-(\ref{Ch5_FinalCond}) given in Section \ref{Chapter5_Preliminaries} (we skip these computations, referring to \cite{agrachev2013,bonnard2006,halanay1966} for details). The relation between the adjoint vector satisfying (\ref{Ch5_Adjoint}) and the above Lagrange multiplier $\tilde \psi_{\tau} = (\psi_{\tau},\psi^0_{\tau})$ is that $(p_{\tau}(\cdot),p^0_{\tau})$ is built so that $p_{\tau}(t^{\tau}_f) = \psi_{\tau}$, $p^0_{\tau} = \psi^0_{\tau}$. We will make use of the result below, which follows from the previous considerations.

\begin{lemma} \label{Ch6_EmmanuelPlane}
Consider the free final time problem (\textbf{OCP})$=$(\textbf{OCP})$_{\tau = 0}$. For any optimal trajectory $x(\cdot)$ of (\textbf{OCP}), the following statements are equivalent:
\begin{itemize}
\item The trajectory $x(\cdot)$ has an unique extremal lift $(x(\cdot),p(\cdot),p^0,u(\cdot))$ whose adjoint $(p(\cdot),p^0)$ is unique up to a multiplicative scalar, which is normal.
\item The cone $\tilde K^{\tau=0}_1(t_f)$ is a half-space of $\ \mathbb{R}^{n+1}$ and $K^{\tau=0}_1(t_f) = \mathbb{R}^n$.
\end{itemize}
\end{lemma}

\subsection{Conic Implicit Function Theorem with Parameters} \label{Ch6_SectConic}

The first step of the proof of Theorem \ref{Ch5_TheoMain} makes use of the procedure detailed in Section \ref{Ch6_SectClassicPMP}. More specifically, we need the needle-like variation formula and the conic implicit function theorem. However, Lemma \ref{Ch6_IFT} is not suited to this situation because we have to take into account the dependence with respect to delays $\tau$. Indeed, the proof of Lemma \ref{Ch6_IFT} is based on Brouwer fixed point theorem (see, e.g., \cite{agrachev2013}) which does not consider continuous dependence with respect to parameters (which, in our case, is represented by $\tau$). Therefore, in this section, we introduce a more general version of the conic implicit function theorem depending on parameters.

When considering delays $\tau$ as a varying parameter, the variation formula provided by Lemma \ref{Ch6_LemmaNeedleLike} holds only for almost every $\tau$, and this, because we need that each $t_i$ be a Lebesgue point of $u_{\tau}(\cdot)$, $u_{\tau}(\cdot - \tau)$ and of $u_{\tau}(\cdot + \tau)$. This leads us to introduce a notion of conic implicit function theorem which, on one hand, ensures a continuous dependence with respect to parameters and, on the other hand, deals with quantities defined uniquely on dense subsets. The notion of differentiability that we need is the following. A function $f : C \subseteq \mathbb{R}^j \rightarrow \mathbb{R}^n$ is said almost everywhere strictly differentiable at some point $x_0 \in C$ whenever there exists a linear continuous mapping $df(x_0) : \mathbb{R}^j \rightarrow \mathbb{R}^n$ such that
$$
f(y) - f(x) = df(x_0)\cdot(y - x) + \| y - x \|g(x,y)
$$
for almost every $x,y \in C$, where $g(x,y)$ tends to 0 as soon as $\| x - x_0 \| + \| y - x_0 \| \xrightarrow{a.e.} 0$.

One may remark that the notion of strictly differentiability and of conic implicit function theorem depending on parameters has already been introduced by \cite{antoine1990}. In our framework, we adapt these results to dense subsets.

\begin{lemma}[Conic Implicit Function Theorem with Parameters] \label{Ch6_IFTP}
Let $C \subseteq \mathbb{R}^j$ be open and convex with non empty interior, of vertex 0, and $F : \mathbb{R}^k_+ \times C \rightarrow \mathbb{R}^n : (\varepsilon,x) \mapsto F(\varepsilon,x)$ be a continuous mapping, for which $F(0,0) = 0$, satisfying the following:
\begin{itemize}
\item For almost every $\varepsilon \in \mathbb{R}^k_+$, $F$ is almost everywhere strictly differentiable with respect to $x$ at 0, and, $\displaystyle \frac{\partial F}{\partial x}(\varepsilon,0)$ is continuous in $\varepsilon$ on a dense subset.
\item For almost every $\varepsilon \in \mathbb{R}^k_+$, the remainder satisfies $g_{\varepsilon}(x,y) \rightarrow 0$ as $(x,y) \xrightarrow{a.e.} 0$, uniformly with respect to $\varepsilon$ on a dense subset.
\item There holds $\displaystyle \frac{\partial F}{\partial x}(0,0)\cdot\textnormal{Cone}(C) = \mathbb{R}^n$.
\end{itemize}
Therefore, there exist $\varepsilon_0 > 0$, a neighborhood $\mathcal{V}$ of 0 in $\mathbb{R}^n$ and a continuous function $h : [0,\varepsilon_0)^k \times \mathcal{V} \rightarrow C$, such that $F(\varepsilon,h(\varepsilon,y)) = y$ for every $\varepsilon \in [0,\varepsilon_0)^k$ and every $y \in \mathcal{V}$.
\end{lemma}

The proof of Lemma \ref {Ch6_IFTP} is done in Appendix \ref{appendixProofSecond}.

\subsection{Proof of Theorem \ref{Ch5_TheoMain}} \label{Ch6_SectMainProof}

From now on, assume that Assumptions $(A)$ hold. Moreover, $(x(\cdot),u(\cdot))$ will denote the (unique) solution of (\textbf{OCP}) and we assume that its related final time $t_f$ is a Lebesgue point of $u(\cdot)$ (if not, as pointed out in Section \ref{Ch6_SectNeedle}, we refer to the approach proposed by \cite{lee1967,gamkrelidze2013}). Finally, without loss of generality, we consider free final time problems (otherwise, the proof is similar, but simpler), introducing further Assumption $(C_2)$ for (\textbf{OCP})$_{\tau}$ with control delays.

\subsubsection{Controllability for (\textbf{OCP})$_{\tau}$} \label{Ch6_Controllability}

For any integer $j \geqslant 1$, fix $0<t_1<\dots<t_j<t_f$ Lebesgue points of control $u(\cdot)$ and $j$ arbitrary values $u_i \in U$. We denote $v|_n$ the first $n$ coordinates of a vector $v \in \mathbb{R}^{n+1}$. For an appropriate small positive real number $\varepsilon_j > 0$, denoting by $\tilde x_{(\varepsilon^0,\varepsilon^1,\varepsilon^2)}(\cdot)$ the trajectory solution of (\ref{Ch6_DynDelayAug}) with delay $\tau = (\varepsilon^0,\varepsilon^1,\varepsilon^2)$ and control $u_{(\varepsilon^0,\varepsilon^1,\varepsilon^2)}(\cdot)$, we define the mapping
$$
\hspace{-0.9ex} \Gamma : B_{\varepsilon_j}(0) \cap (\mathbb{R}^3_+ \times \mathbb{R} \times \mathbb{R}^j_+) \rightarrow \mathbb{R}^n : (\varepsilon^0,\varepsilon^1,\varepsilon^2,\delta,\eta_1,\dots,\eta_j) \mapsto ( \tilde x^{\pi}_{(\varepsilon^0,\varepsilon^1,\varepsilon^2)}(t_f + \delta) - \tilde x(t_f) )\big|_n
$$
which, thanks to Assumption $(A_2)$ and by continuity with respect to initial data, is well-defined and continuous. Moreover, $\Gamma(0,\dots,0) = 0$ and
$$
\Gamma(\varepsilon^0,\dots,\eta_j) = ( \tilde x^{\pi}_{(\varepsilon^0,\varepsilon^1,\varepsilon^2)}(t_f + \delta) - \tilde x_{(\varepsilon^0,\varepsilon^1,\varepsilon^2)}(t_f) )\big|_n + ( \tilde x_{(\varepsilon^0,\varepsilon^1,\varepsilon^2)}(t_f) - \tilde x(t_f) )\big|_n .
$$
From Lemma \ref{Ch6_LemmaNeedleLike} and a recursive use of the needle-like variation formula (see Section \ref{Ch6_SectProofPMPDelay}), for almost every $(\varepsilon^0,\varepsilon^1,\varepsilon^2)$ small enough, $\Gamma$ is almost everywhere strictly differentiable w.r.t. $(\delta,\eta_1,\dots,\eta_j)$ at 0, $\displaystyle \frac{\partial \Gamma}{\partial (\delta,\eta_1,\dots,\eta_j)}(\varepsilon^0,\varepsilon^1,\varepsilon^2,0)$ is continuous w.r.t. $(\varepsilon^0,\varepsilon^1,\varepsilon^2)$ on a dense subset and, moreover, the remainder of the related Taylor expansion converges to zero uniformly w.r.t. $(\varepsilon^0,\varepsilon^1,\varepsilon^2)$ on a dense subset.

From Assumption $(A_3)$, the unique extremal lift of $x(\cdot)$ is normal, hence, it follows from Lemma \ref{Ch6_EmmanuelPlane} that $\textnormal{Int} \; K^{\tau = 0}_1(t_f) = \mathbb{R}^n$. We recall that we consider only either optimal control problems with pure state delays or control-affine optimal control problems. Therefore, thanks to Remark \ref{Ch6_RemarkVariations}, there exist a real number $\delta$, an integer $j \geqslant 1$ and a variation $\pi = \{ t_1,\dots,t_j,\eta_1,\dots,\eta_j,u_1,\dots,u_j\}$ such that
$$
\frac{\partial \Gamma}{\partial (\delta,\eta_1,\dots,\eta_j)}(0)\cdot(\mathbb{R} \times \mathbb{R}^j_+) = \textnormal{Int} \; K^{\tau = 0}_1(t_f) = \mathbb{R}^n .
$$
At this step, Lemma \ref{Ch6_IFTP} implies the existence of $\varepsilon_0 > 0$ such that, for every $\tau = (\tau^0,\tau^1,\tau^2) \in [0,\varepsilon_0)^3$, there exist a real $\delta(\tau)$ and positive reals $\eta_1(\tau),\dots,\eta_j(\tau)$ such that $\Gamma(\tau^0,\tau^1,\tau^2,\delta(\tau),\eta_1(\tau),\dots,\eta_j(\tau)) = 0$. Moreover, $\delta(\tau)$, $\eta_1(\tau)$, $\dots$,$\eta_j(\tau)$ are continuous with respect to $\tau$. From Assumption $(A_4)$, it follows that, for every $\tau = (\tau^0,\tau^1,\tau^2) \in [0,\varepsilon_0)^3$, the subset $M_f$ is reachable for the dynamics of (\textbf{OCP})$_{\tau}$, in a final time $t^{\tau}_f \in [0,b]$, by using control $u^{\pi}_{(\tau^1,\tau^2)}(\cdot) \in L^{\infty}([0,t^{\tau}_f],U)$. \\

We have proved that, for every $\tau = (\tau^0,\tau^1,\tau^2) \in (0,\varepsilon_0)^3$, (\textbf{OCP})$_{\tau}$ is controllable. Remark that this argument still holds for (\textbf{OCP})$_{\tau}$ with pure state delays.

\subsubsection{Existence of Optimal Controls for (\textbf{OCP})$_{\tau}$} \label{Ch6_ExistenceSect}

We focus first on the existence of an optimal control for (\textbf{OCP})$_{\tau}$, for every  $\tau \in (0,\varepsilon_0)^3$. No other assumptions but $(A)$ and $(C_1)$ are considered. In particular, mappings $f$ and $f^0$ are affine in the two control variables. Thanks to this property, existence can be established by using the arguments in \cite[Theorem 2]{bonalli2017}. However, we prefer to develop the usual Filippov's scheme \cite{filippov1962} (following \cite{lee1967,trelat2008}) to highlight the difficulty in applying this procedure to more general systems (in particular, see Remark \ref{Ch6_RemarkGuinn}). Even if problems (\textbf{OCP})$_{\tau}$ with control and state delays are considered, we assume to have free final time just to use the same approach for problem with pure state delays. \\

Fix $\tau = (\tau^0,\tau^1,\tau^2) \in (0,\varepsilon_0)^3$ and let
$$
\alpha = \inf_{u \in \mathcal{U}^{\tau}_U} C_{\tau}(t_f(u),u) = \int_{0}^{t_f(u)} f^0(t,t-\tau^0,x(t),x(t-\tau^1),u(t),u(t-\tau^2)) \; \mathrm{d}t .
$$
Consider now a minimizing sequence of trajectories $x_k(\cdot)$ associated to $u_k(\cdot)$, that is $C_{\tau}(t_f(u_k),u_k) \rightarrow \alpha$ when $k \rightarrow \infty$ and define
$$
\tilde F_k(t) = \tilde f(t,t-\tau^0,x_k(t),x_k(t-\tau^1),u_k(t),u_k(t-\tau^2))
$$
for almost every $t \in [0,t_f(u_k)]$. By Assumption $(A_4)$, we can extend $\tilde F_k(\cdot)$ by zero on $(t_f(u_k),b]$ so that $(\tilde F_k(\cdot))_{k \in \mathbb{N}}$ is bounded in $L^{\infty}([0,b],\mathbb{R}^{n+1})$. Therefore, up to some subsequence, $(\tilde F_k(\cdot))_{k \in \mathbb{N}}$ converges to some $\tilde F(\cdot) = (F(\cdot),F^0(\cdot)) \in L^{\infty}([0,b],\mathbb{R}^{n+1})$ for the weak star topology of $L^{\infty}$. On the other hand, up to some subsequence, the sequence $(t_f(u_k))_{k \in \mathbb{N}}$ converges to some $t^{\tau}_f \geqslant 0$. Then, for every $t \in [-\Delta,t^{\tau}_f]$, define
\begin{equation} \label{Ch6_ExistenceTraj}
x_{\tau}(t) = \phi^1(t) \mathds{1}_{[-\Delta,0)}(t) + \mathds{1}_{[0,t^{\tau}_f]}(t) \bigg( \phi^1(0) + \int_{0}^{t} F(s) \; ds \bigg) .
\end{equation}
Now, $x_{\tau}(\cdot)$ is absolutely continuous and, considering continuous extensions, $(x_k(\cdot))_{k \in \mathbb{N}}$ converges pointwise to $x_{\tau}(\cdot)$ within $[-\Delta,t^{\tau}_f]$. Moreover, by Assumptions $(A_1)$, $(A_4)$ and the Arzel\`{a}-Ascoli theorem, up to some subsequence, $(x_k(\cdot))_{k \in \mathbb{N}}$ converges to $x_{\tau}(\cdot)$, uniformly in $[-\Delta,t^{\tau}_f]$. From the compactness of $M_f$, we have $x_{\tau}(t^{\tau}_f) \in M_f$. \\

In the next paragraph, we show that $x_{\tau}(\cdot)$ comes from a control in $\mathcal{U}^{\tau}_{t^{\tau}_f,U}$.

For almost every $t \in [0,t_f(u_k)]$, set
$$
\tilde H_k(t) = \tilde f(t,t-\tau^0,x_{\tau}(t),x_{\tau}(t-\tau^1),u_k(t),u_k(t-\tau^2))
$$
and, if $t_f(u_k) + \tau^2 < t^{\tau}_f$, extend it by 0 in $ (t_f(u_k),b]$. At this step, we need to introduce several structures to deal with the presence of the control delay $\tau^2$. First, let
$$
\beta = \max \Big\{ \ |f^0(t,s,x,y,u,v)| \ : \ -\Delta \leqslant t,s \leqslant b \ , \ \| (x,y) \| \leqslant b  \ , \ (u,v) \in U^2 \ \Big\} > 0
$$
and $N \in \mathbb{N}$ such that $N \tau^2 \leqslant t^{\tau}_f < (N+1) \tau^2$. Considering continuous extensions, we see that $x_{\tau}(\cdot)$ is well-defined in $[-\Delta,(N+1) \tau^2]$. We define
\begingroup
\footnotesize
\textcolor{black}{\begin{equation} \label{Ch6_RedGuinn}
\begin{split}
\tilde G&(t,u^1,\dots,u^{N+1},
\gamma^1,\dots,\gamma^{N+1}) \\
&= \left( \begin{array}{c}
f(t,t-\tau^0,x_{\tau}(t),x_{\tau}(t-\tau^1),u^1,\phi^2(t-\tau^2)) \\
f^0(t,t-\tau^0,x_{\tau}(t),x_{\tau}(t-\tau^1),u^1,\phi^2(t-\tau^2)) + \gamma^1 \\
f(t+\tau^2,t+\tau^2-\tau^0,x_{\tau}(t+\tau^2),x_{\tau}(t+\tau^2-\tau^1),u^2,u^1) \\
f^0(t+\tau^2,t+\tau^2-\tau^0,x_{\tau}(t+\tau^2),x_{\tau}(t+\tau^2-\tau^1),u^2,u^1) + \gamma^2 \\
\dots \\
f(t+N\tau^2,t+N\tau^2-\tau^0,x_{\tau}(t+N\tau^2),x_{\tau}(t+N\tau^2-\tau^1),u^{N+1},u^N) \\
f^0(t+N\tau^2,t+N\tau^2-\tau^0,x_{\tau}(t+N\tau^2),x_{\tau}(t+N\tau^2-\tau^1),u^{N+1},u^N) + \gamma^{N+1}
\end{array} \right)
\end{split}
\end{equation}}
\endgroup
almost everywhere in $[0,\tau^2]$, and
\begingroup
\footnotesize
$$
\tilde V_{\beta}(t) = \bigg\{ \ \tilde G(t,u^1,\dots,u^{N+1},\gamma^1,\dots,\gamma^{N+1}) \ : \ (u^1,\dots,u^{N+1}) \in U^{N+1} \ , \ \forall i=1, \ldots, N+1 \ : \ \gamma^i \geqslant 0 \ ,
$$
$$
\hspace{50pt} |f^0(t,t-\tau^0,x_{\tau}(t),x_{\tau}(t-\tau^1),u^1,\phi^2(t-\tau^2)) + \gamma^1| \leqslant \beta
$$
$$
\hspace{45pt} \forall i=1, \ldots, N \ : \ |f^0(t+i\tau^2,t+i\tau^2-\tau^0,x_{\tau}(t+i\tau^2),x_{\tau}(t+i\tau^2-\tau^1),u^{i+1},u^{i}) + \gamma^{i+1}| \leqslant \beta \ \bigg\} .
$$
\endgroup
Thanks to Assumption $(A_1)$, $\tilde V_{\beta}(t)$ is compact for the standard topology of \begingroup\small$\mathbb{R}^{(n+1)(N+1)}$\endgroup. Moreover, Assumptions $(A_1)$ and $(C_1)$ ensure that $\tilde V_{\beta}(t)$ is convex. We have that
$$
\tilde{\mathcal{V}} = \Big\{ \ \tilde G(\cdot) \in L^2([0,\tau^2],\mathbb{R}^{(n+1)(N+1)}) \ : \ \tilde G(t) \in \tilde V_{\beta}(t) \ , \ \textnormal{ a.e. } [0,\tau^2] \ \Big\} .
$$
is convex and closed in $L^2([0,\tau^2],\mathbb{R}^{(n+1)(N+1)})$ for the strong topology of $L^2$, and therefore, it is convex and closed in $L^2([0,\tau^2],\mathbb{R}^{(n+1)(N+1)})$ for the weak topology of $L^2$. At this step, for every $i=0,\dots,N$, denote
$$
\tilde G^{i+1}_k(t) = \tilde f(t+i\tau^2,t+i\tau^2-\tau^0,x_{\tau}(t+i\tau^2),x_{\tau}(t+i\tau^2-\tau^1),u_k(t+i\tau^2),u_k(t+(i-1)\tau^2))
$$
and $\tilde G_k(t) = (\tilde G^{1}_k(t),\dots,\tilde G^{N+1}_k(t))$. Therefore, $\tilde G_k(\cdot) \in \tilde{\mathcal{V}}$ for every $k \in \mathbb{N}$. Moreover, since $(\tilde G_k(\cdot))_{k \in \mathbb{N}}$ is bounded in $L^2([0,\tau^2],\mathbb{R}^{(n+1)(N+1)})$, up to some subsequence, it converges for the weak topology of $L^2$ to a function $\tilde G(\cdot)$ that necessarily belongs to $\tilde{\mathcal{V}}$. Therefore, for almost every $t \in [0,\tau^2]$ and $i=1,\dots,N+1$, there exist points $u^i_{\tau}(t) \in U$, and scalar $\gamma^i_{\tau}(t) \geqslant 0$ such that
\begingroup
$$
\tilde G^{1}(t) = \left( \begin{array}{c}
f(t,t-\tau^0,x_{\tau}(t),x_{\tau}(t-\tau^1),u^{1}_{\tau}(t),\phi^2(t-\tau^2)) \\
f^0(t,t-\tau^0,x_{\tau}(t),x_{\tau}(t-\tau^1),u^{1}_{\tau}(t),\phi^2(t-\tau^2)) + \gamma^1_{\tau}(t)
\end{array} \right)
$$
\endgroup
and, for every $i=1,\dots,N$,
\begingroup
\footnotesize
$$
\tilde G^{i+1}(t) = \left( \begin{array}{c}
f(t+i\tau^2,t+i\tau^2-\tau^0,x_{\tau}(t+i\tau^2),x_{\tau}(t+i\tau^2-\tau^1),u^{i+1}_{\tau}(t),u^{i}_{\tau}(t)) \\
f^0(t+i\tau^2,t+i\tau^2-\tau^0,x_{\tau}(t+i\tau^2),x_{\tau}(t+i\tau^2-\tau^1),u^{i+1}_{\tau}(t),u^{i}_{\tau}(t)) + \gamma^{i+1}_{\tau}(t)
\end{array} \right) .
$$
\endgroup
Moreover, since $U$ is compact, functions $u^i_{\tau}(\cdot)$, $\gamma^i_{\tau}(\cdot)$ can be chosen to be measurable on $[0,\tau^2]$ using a measurable selection lemma (see, e.g., \cite[Lemma 3A, page 161]{lee1967}).

At this step, we come back to the whole interval $[-\tau^2,t^{\tau}_f]$. For this, set
\begingroup
$$
u_{\tau}(t) = \bigg\{ \begin{array}{cc}
\phi^2(t) & t \in [-\tau^2,0] \ , \\
u^{i+1}_{\tau}(t-i\tau^2) & t \in [i\tau^2,(i+1)\tau^2] \ , \ i=0,\dots,N
\end{array}
$$
$$
\gamma_{\tau}(t) = \gamma^{i+1}_{\tau}(t-i\tau^2) \quad t \in [i\tau^2,(i+1)\tau^2] \ , \ i=0,\dots,N
$$
\endgroup
which are measurable functions in $[-\tau^2,t^{\tau}_f]$, and let
$$
\tilde H(t) = \left( \begin{array}{c}
f(t,t-\tau^0,x_{\tau}(t),x_{\tau}(t-\tau^1),u_{\tau}(t),u_{\tau}(t-\tau^2)) \\
f^0(t,t-\tau^0,x_{\tau}(t),x_{\tau}(t-\tau^1),u_{\tau}(t),u_{\tau}(t-\tau^2)) + \gamma_{\tau}(t)
\end{array} \right) .
$$
From the weak star convergence in $L^{\infty}$ of $(\tilde G_k(\cdot))_{k \in \mathbb{N}}$ towards $\tilde G(\cdot)$, it follows immediately that $(\tilde H_k(\cdot))_{k \in \mathbb{N}}$ converges to $\tilde H(\cdot)$ for the weak topology of $L^2$. Furthermore, from the differentiability of $\tilde f$ w.r.t. $(x,y)$, the compactness of $U$ and the dominated convergence theorem, one has
$$
\underset{k \rightarrow \infty}{\lim} \ \int_{0}^{t^{\tau}_f} \Big( \tilde F_k(t) - \tilde H_k(t) \Big) \cdot \varphi(t) \; \mathrm{d}t = 0
$$
for every map $\varphi(\cdot) \in L^2([0,t^{\tau}_f],\mathbb{R}^{n+1})$, from which $\tilde H = \tilde F$ almost everywhere in $[0,t^{\tau}_f]$.

Combining (\ref{Ch6_ExistenceTraj}) with all the previous results, we obtain
\begingroup
\small
$$
\hspace{-3.5pt} x_{\tau}(t) = \phi^1(t) \mathds{1}_{[-\Delta,0)}(t) + \mathds{1}_{[0,t^{\tau}_f]}(t) \bigg( \phi^1(0) + \int_{0}^{t} f(t,t-\tau^0,x_{\tau}(t),x_{\tau}(t-\tau^1),u_{\tau}(t),u_{\tau}(t-\tau^2)) \; ds \bigg)
$$
\endgroup
which proves that the measurable function $u_{\tau}(\cdot)$ is an admissible control for (\textbf{OCP})$_{\tau}$. \\

It remains to show that control $u_{\tau}(\cdot)$ is optimal for (\textbf{OCP})$_{\tau}$. For this, from what we showed above and by definition of weak star convergence, we have
$$
C_{\tau}(t_f(u_k),u_k) \rightarrow \int_{0}^{t^{\tau}_f} \big( f^0(t,t-\tau^0,x_{\tau}(t),x_{\tau}(t-\tau^1),u_{\tau}(t),u_{\tau}(t-\tau^2)) + \gamma_{\tau}(t) \big) \; \mathrm{d}t .
$$
Since $\gamma_{\tau}(\cdot)$ takes only non-negative values, one finally has
$$
\int_{0}^{t^{\tau}_f} f^0(t,t-\tau^0,x_{\tau}(t),x_{\tau}(t-\tau^1),u_{\tau}(t),u_{\tau}(t-\tau^2)) \; \mathrm{d}t \leqslant \alpha \leqslant C_{\tau}(t_f(v),v)
$$
for every $v(\cdot) \in \mathcal{U}^{\tau}_{U}$. Therefore, $\gamma_{\tau}(\cdot)$ is necessarily zero and the conclusion follows. \\

Now, we consider problems (\textbf{OCP})$_{\tau}$ with pure state delays. It is clear that, if Assumption $(B_2)$ holds, one can proceed with the same argument as above (which is nothing else but the usual Filippov's scheme, see, e.g., \cite{filippov1962}) for the existence of optimal controls. In this case, Guinn's reduction (\ref{Ch6_RedGuinn}) is not needed.

\begin{remark} \label{Ch6_RemarkGuinn}
Guinn's reduction (\ref{Ch6_RedGuinn}) converts the dynamics with control delays into a new dynamics without control delays but with a larger number of variables. It is clear from the context, that the natural assumption to provide the existence of optimal controls for generic nonlinear dynamics is the convexity of system (\ref{Ch6_RedGuinn}) for every $N \in \mathbb{N}$ (since we make delays vary), which is a very strong assumption. The proof of Lemma 2.1 in \cite{nababan1979} does not work under the weaker assumption of convexity of the epigraph of the extended dynamics.
\end{remark}

\subsubsection{Convergence of Optimal Controls and Trajectories for (\textbf{OCP})$_{\tau}$} \label{Ch6_SectionConvTraj}

We start by considering problems (\textbf{OCP})$_{\tau}$ with pure state delays, by assuming that Assumption $(B_2)$ holds. In this case, the classical way to proceed consists in reproducing and adapting the convexity Filippov's scheme used in the previous section for the existence of optimal controls (see, e.g., \cite{haberkorn2011}). \\

Let $(\tau_k)_{k \in \mathbb{N}} = ((\tau^0_k,\tau^1_k,0))_{k \in \mathbb{N}} \subseteq (0,\varepsilon_0)^2 \times \{ 0 \}$ be an arbitrary sequence converging to 0 as $k$ tends to $\infty$ and let $(x_{\tau_k}(\cdot),u_{\tau_k}(\cdot))$ be an optimal solution for (\textbf{OCP})$_{\tau_k}$ with final time $t^{\tau_k}_f(u_{\tau_k})$. Since $t^{\tau_k}_f(u_{\tau_k}) \in [0,b]$, up to some subsequence, the sequence $(t^{\tau_k}_f)_{k \in \mathbb{N}} = (t^{\tau_k}_f(u_{\tau_k}))_{k \in \mathbb{N}}$ converges to some $\bar t_f \in [0,b]$. Since $M_f$ is compact, up to some subsequence, the sequence $(x_{\tau_k}(t^{\tau_k}_f))_{k \in \mathbb{N}} \subseteq M_f$ converges to some point in $M_f$.

For every integer $k$ and almost every $t \in [0,t^{\tau_k}_f]$, set
\begingroup
$$
\hspace{0pt} \tilde G_k(t) = \bigg( \tilde f(t,t-\tau^0_k,x_{\tau_k}(t),x_{\tau_k}(t-\tau^1_k),u_{\tau_k}(t)) , \frac{\partial \tilde f}{\partial x}(t,t-\tau^0_k,x_{\tau_k}(t),x_{\tau_k}(t-\tau^1_k),u_{\tau_k}(t)) ,
$$
$$
\hspace{200pt} \frac{\partial \tilde f}{\partial y}(t,t-\tau^0_k,x_{\tau_k}(t),x_{\tau_k}(t-\tau^1_k),u_{\tau_k}(t)) \bigg) .
$$
\endgroup
Thanks to Assumption $(A_4)$, we extend $\tilde G_k(\cdot)$ by zero on $(t^{\tau_k}_f,b]$. Assumptions $(A_1)$ and $(A_4)$ imply that the sequence $(\tilde G_k(\cdot))_{k \in \mathbb{N}}$ is bounded in $L^{\infty}$, then, up to some subsequence, it converges to some $\tilde G(\cdot) = (G(\cdot),G^0(\cdot),G_x(\cdot),G_y(\cdot)) \in L^{\infty}([0,b],\mathbb{R}^{n+1})$ for the weak star topology of $L^{\infty}$. Exploiting the weak star convergence of $L^{\infty}$ (and using $\mathds{1}_{[\bar t_f,b]} \tilde G$ as test function), we get that $\tilde G(t) = 0$ for almost every $t \in [\bar t_f,b]$. From this, for every $t \in [0,\bar t_f]$, denote
\begin{equation} \label{Ch6_TrajZero}
\bar x(t) = \phi^1(t) \mathds{1}_{[-\Delta,0)}(t) + \mathds{1}_{[0,\bar t_f]}(t) \bigg( \phi^1(0) + \int_{0}^{t} G(s) \; ds \bigg) .
\end{equation}
Clearly, $\bar x(\cdot)$ is absolutely continuous and $\bar x(t) = \underset{k \rightarrow \infty}{\lim} x_{\tau_k}(t)$ pointwise in $[-\Delta,\bar t_f]$. By Assumptions $(A_1)$, $(A_4)$ and the Arzel\`{a}-Ascoli theorem, up to some subsequence, $(x_{\tau_k}(\cdot))_{k \in \mathbb{N}}$ converges to $\bar x(\cdot)$, uniformly in $[-\Delta,\bar t_f]$, and we have $\bar x(\bar t_f) \in M_f$. \\

In the next paragraph, we prove that there exists a control $\bar u(\cdot) \in L^{\infty}([0,\bar t_f],U)$ such that $\bar x(\cdot)$ is an admissible trajectory for (\textbf{OCP}) associated with this control $\bar u(\cdot)$.

Using the definition of $\beta$ given in Section \ref{Ch6_ExistenceSect}, for $t \in [0,\bar t_f]$, consider the set
\begingroup
\footnotesize
$$
\hspace{0pt} \tilde Z_{\beta}(t) = \bigg\{ \left(f(t,t,\bar x(t),\bar x(t),u),f^0(t,t,\bar x(t),\bar x(t),u)+\gamma,\frac{\partial \tilde f}{\partial x}(t,t,\bar x(t),\bar x(t),u),\frac{\partial \tilde f}{\partial y}(t,t,\bar x(t),\bar x(t),u)\right) :
$$
$$
\hspace{190pt} u \in U \ , \ \gamma \geqslant 0 \ , \ | f^0(t,t,\bar x(t),\bar x(t),u) + \gamma | \leqslant \beta \bigg\} .
$$
\endgroup
Thanks to Assumption $(B_2)$, the set $\tilde Z_{\beta}(t)$ is compact and convex in $\mathbb{R}^{n+1}$. We have the following statements:
\begin{itemize}
\item From the convexity and the compactness of $\tilde Z_{\beta}(t)$, for $\delta > 0$ and $t \in [0,\bar t_f]$,
$$
\tilde Z^{\delta}_{\beta}(t) = \bigg\{ \ x \in \mathbb{R}^{n+1} \ : \ d(x,\tilde Z_{\beta}(t)) \leqslant \delta \ \bigg\} , \ \textnormal{where} \ \  d(x,A) = \underset{y \in A}{\inf} \ \| x - y \|
$$
is convex and compact for the standard topology of $\mathbb{R}^{n+1}$.

\item For every $\delta > 0$, the set
$$
\tilde{\mathcal{Z}}_{\delta} = \bigg\{ \ \tilde  F(\cdot) \in L^2([0,\bar t_f],\mathbb{R}^{n+1}) \ : \ \tilde F(t) \in \tilde  Z^{\delta}_{\beta}(t) \ \textnormal{for almost every} \ t \in [0,\bar t_f] \ \bigg\}
$$
is convex and closed in $L^2([0,\bar t_f],\mathbb{R}^{n+1})$ for the strong topology of $L^2$. Then, it is closed in $L^2([0,\bar t_f],\mathbb{R}^{n+1})$ for the weak topology of $L^2$.

Convexity is obvious from the previous statement. Let $(\tilde F_k(\cdot))_{k \in \mathbb{N}} \subseteq \tilde{\mathcal{Z}}_{\delta}$ such that $\tilde F_k(\cdot) \xrightarrow{L^2} \tilde F(\cdot)$. Therefore, $\tilde F(\cdot) \in L^2([0,\bar t_f],\mathbb{R}^{n+1})$ and there exists a subsequence such that $\tilde F_{k_m}(\cdot) \xrightarrow{a.e} \tilde F(\cdot)$. Since $\tilde Z^{\delta}_{\beta}(t)$ is closed for the standard topology of $\mathbb{R}^{n+1}$, a.e. in $t \in [0, \bar t_f]$, we have that $\displaystyle \tilde F(t) = \lim_{m \rightarrow \infty} \tilde F_{k_m}(t) \in \tilde Z^{\delta}_{\beta}(t)$ and the statement follows.

\item For every $\delta > 0$, there exists $k_{\delta} \in \mathbb{N}$ such that, if $k \geqslant k_{\delta}$, then $\tilde G_k(\cdot) \in \tilde{\mathcal{Z}}_{\delta}$.

Indeed, thanks to Assumptions $(A_1)$, $(A_4)$, mappings $f$, $f^0$ are globally Lipschitz within $[-\Delta,\bar t_f]^2 \times \overline{B^{2 n}_b(0)} \times U$ and, for almost every $t \in [0,\bar t_f]$
$$
\displaystyle \underset{z \in \tilde Z_{\beta}(t)}{\inf} \| \tilde G_k(t) - z \| \leqslant \tilde C \Big( \| x_{\tau_k}(t) - \bar x(t) \| + \| x_{\tau_k}(t - \tau^1_k) - \bar x(t) \| + \tau^0_k \Big)
$$
where $\tilde C > 0$ is a suitable constant, which is independent from $t$. The conclusion follows from the uniform convergence of $(x_{\tau_k}(\cdot))_{k \in \mathbb{N}}$ towards $\bar x(\cdot)$.
\end{itemize}
Using the closeness of $\tilde{\mathcal{Z}}_{\delta}$ with respect to the weak topology of $L^2$, we infer that $\tilde G(\cdot) \in \tilde{\mathcal{Z}}_{\delta}$ for every $\delta > 0$. This implies $\tilde G(\cdot) \in \underset{j \in \mathbb{N}}{\cap} \tilde{\mathcal{Z}}_{1/j} \subseteq \tilde{\mathcal{Z}}_0$.

We have obtained that, a.e. in $t \in [0,\bar t_f]$, there exist $\bar u(t) \in U$, $\bar \gamma(t) \geqslant 0$ such that
\begingroup
\begin{equation} \label{Ch6_GConvergence}
\hspace{-40pt} \tilde G(t) = \bigg( f(t,t,\bar x(t),\bar x(t),\bar u(t)) \ , \ f^0(t,t,\bar x(t),\bar x(t),\bar u(t)) + \bar \gamma(t) \ ,
\end{equation}
$$
\hspace{150pt} \frac{\partial \tilde f}{\partial x}(t,t,\bar x(t),\bar x(t),\bar u(t)) \ , \ \frac{\partial \tilde f}{\partial y}(t,t,\bar x(t),\bar x(t),\bar u(t)) \bigg) .
$$
\endgroup
Moreover, since $U$ is compact, functions $\bar u(\cdot)$, $\bar \gamma(\cdot)$ can be chosen to be measurable on $[0,\bar t_f]$ using a measurable selection lemma (see, e.g., \cite[Lemma 3A, page 161]{lee1967}). Combining (\ref{Ch6_GConvergence}) with (\ref{Ch6_TrajZero}) provides
$$
\bar x(t) = \phi^1(t) \mathds{1}_{[-\Delta,0)}(t) + \mathds{1}_{[0,\bar t_f]}(t) \bigg( \phi^1(0) + \int_{0}^{t} f(s,s,\bar x(s),\bar x(s),\bar u(s)) \; ds \bigg)
$$
which proves that the function $\bar u(\cdot)$ is an admissible control for (\textbf{OCP}). \\

In order to conclude, it remains to show that $\bar t_f = t_f$, $\bar u(\cdot) = u(\cdot)$ and $\bar x(\cdot) = x(\cdot)$.

First, the previous argument shows that
$$
C_{\tau_k}(t^{\tau_k}_f,u_{\tau_k}) \rightarrow C_{0}(\bar t_f,\bar u) + \int_{0}^{\bar t_f} \bar \gamma(t) \; \mathrm{d}t .
$$
Thanks to the construction of the mapping $\Gamma$ in Section \ref{Ch6_Controllability}, for every integer $k$, there exists a sequence $(t^k_f,v_k(\cdot),y_k(\cdot))_{k \in \mathbb{N}}$, respectively of final times, of admissible controls and of trajectories for (\textbf{OCP})$_{\tau_k}$, which converges to $(t_f,u(\cdot),x(\cdot))$ (for the evident topologies) as $k$ tends to $\infty$. Thanks to the optimality of each $u_{\tau_k}(\cdot)$, we have $C_{\tau_k}(t^{\tau_k}_f,u_{\tau_k}) \leqslant C_{\tau_k}(t^k_f,v_k)$ and, since $\bar \gamma(t) \geqslant 0$, passing to the limit gives $ C_{0}(\bar t_f,\bar u) \leqslant  C_{0}(t_f,u)$. From Assumption $(A_2)$, we infer $\bar t_f = t_f$, $\bar u(\cdot) = u(\cdot)$, $\bar x(\cdot) = x(\cdot)$.

Remark that, from the previous argument, the following weak convergences hold
\begingroup
\begin{eqnarray} \label{Ch6_DerDynConv}
\begin{cases}
\displaystyle \frac{\partial \tilde f}{\partial x}(\cdot,\cdot-\tau^0_k,x_{\tau_k}(\cdot),x_{\tau_k}(\cdot-\tau^1_k),u_{\tau_k}(\cdot)) \overset{(L^{\infty})^*}{\rightharpoonup} \frac{\partial \tilde f}{\partial x}(\cdot,\cdot,x(\cdot),x(\cdot),u(\cdot)) \medskip \\
\displaystyle \frac{\partial \tilde f}{\partial y}(\cdot,\cdot-\tau^0_k,x_{\tau_k}(\cdot),x_{\tau_k}(\cdot-\tau^1_k),u_{\tau_k}(\cdot)) \overset{(L^{\infty})^*}{\rightharpoonup} \frac{\partial \tilde f}{\partial y}(\cdot,\cdot,x(\cdot),x(\cdot),u(\cdot)) \ .
\end{cases}
\end{eqnarray}
\endgroup

\vspace{10pt}

Consider now (\textbf{OCP})$_{\tau}$ with control and state delays satisfying Assumption $(C_1)$. Since the mappings are control-affine, the previous argument simplifies considerably, because we can transpose the weak convergence directly on controls. Adapting these proofs to much more general systems is very challenging. We adopt free final time to show that, for this step, no problems arise if Assumption $(C_2)$ does not hold. \\

Let $(\tau_k)_{k \in \mathbb{N}} = ((\tau^0_k,\tau^1_k,\tau^2_k))_{k \in \mathbb{N}} \subseteq (0,\varepsilon_0)^3$ an arbitrary sequence of delays converging to 0 as $k$ tends to $\infty$ and let $(x_{\tau_k}(\cdot),u_{\tau_k}(\cdot))$ be an optimal solution of (\textbf{OCP})$_{\tau_k}$ with final time $t^{\tau_k}_f(u_{\tau_k})$. Since $t^{\tau_k}_f(u_{\tau_k}) \in [0,b]$, up to some subsequence, the sequence of final times $(t^{\tau_k}_f)_{k \in \mathbb{N}} = (t^{\tau_k}_f(u_{\tau_k}))_{k \in \mathbb{N}}$ converges to some $\bar t_f \in [0,b]$. Since $M_f$ is compact, up to some subsequence, the sequence
$(x_{\tau_k}(t^{\tau_k}_f))_{k \in \mathbb{N}} \subseteq M_f$ converges to a point in $M_f$. \\

On the other hand, thanks to Assumption $(A_1)$, the sequence $(u_{\tau_k}(\cdot))_{k \in \mathbb{N}}$ is bounded in $L^2([-\Delta,\bar t_f],\mathbb{R}^m)$. Therefore, up to some subsequence, $(u_{\tau_k}(\cdot))_{k \in \mathbb{N}}$ converges to some $\bar u(\cdot) \in L^2([-\Delta,\bar t_f],\mathbb{R}^m)$ for the weak topology of $L^2$. More precisely, it holds $\bar u(\cdot) \in L^{\infty}([-\Delta,\bar t_f],U)$. Indeed, $(u_{\tau_k}(\cdot))_{k \in \mathbb{N}} \subseteq L^2([-\Delta,\bar t_f],U)$ and, thanks to Assumption $(A_1)$, the set $L^2([-\Delta,\bar t_f],U)$ is closed and convex for the strong topology of $L^2$. Therefore, it is closed and convex for the weak topology of $L^2$, from which $\bar u(\cdot) \in L^2([-\Delta,\bar t_f],U) \subseteq L^{\infty}([-\Delta,\bar t_f],U)$ (the last inclusion still follows from $(A_1)$).

At this step, one crucial result is the weak convergence in $L^2$ of the sequence $(u_{\tau_k}(\cdot - \tau^2_k))_{k \in \mathbb{N}}$ towards control $\bar u(\cdot)$. To see this, consider the shift operator
$$
S_{\tau^2} : L^2(\mathbb{R},\mathbb{R}^m) \rightarrow L^2(\mathbb{R},\mathbb{R}^m) : \Big( t \mapsto \phi(t) \Big) \mapsto \Big( t \mapsto \phi(t - \tau^2) \Big) .
$$
Using the dominated convergence theorem, it is clear that, for every $\phi(\cdot) \in L^2(\mathbb{R},\mathbb{R}^m)$, it holds $\| S_{\tau^2} \phi - \phi \|_{L^2} \rightarrow 0$ as soon as $\tau^2 \rightarrow 0$. At this point, extend $u_{\tau_k}(\cdot)$, $u_{\tau_k}(\cdot-\tau^2_k)$ and $\bar u(\cdot)$ by zero outside of $[-\Delta,\bar t_f]$. For every $\varphi(\cdot) \in L^2(\mathbb{R},\mathbb{R}^m)$, one obtains
\begingroup
\small
$$
\int_{0}^{\bar t_f} (u_{\tau_k}(t-\tau_k) - \bar u(t)) \cdot \varphi(t) \; \mathrm{d}t = \int_{\mathbb{R}} (u_{\tau_k}(t) - \bar u(t)) \cdot \Big( S_{-\tau^2_k} \varphi \Big)(t) \; \mathrm{d}t + \int_{\mathbb{R}} (S_{\tau^2_k} \bar u - \bar u)(t) \cdot \varphi(t) \; \mathrm{d}t
$$
\begin{equation} \label{Ch6_ShiftExpression}
= \int_{0}^{\bar t_f} (u_{\tau_k}(t) - \bar u(t)) \cdot \varphi(t) \; \mathrm{d}t + \int_{\mathbb{R}} (u_{\tau_k}(t) - \bar u(t)) \cdot \Big( S_{-\tau^2_k} \varphi - \varphi \Big)(t) \; \mathrm{d}t + \int_{\mathbb{R}} (S_{\tau^2_k} \bar u - \bar u)(t) \cdot \varphi(t) \; \mathrm{d}t
\end{equation}
\endgroup
which converges to 0, providing the weak convergence in $L^2$ of $(u_{\tau_k}(\cdot - \tau^2_k))_{k \in \mathbb{N}}$ to $\bar u(\cdot)$. \\

We can now show that, under Assumption $(C_1)$, the trajectory arising from control $\bar u(\cdot)$ is admissible for problem (\textbf{OCP}), proceeding as follows. First, remark that, up to continuous extensions, for every $k$, we have
\begingroup
\footnotesize
\begin{equation} \label{Ch6_ExprDynDelay}
\hspace{-5.5pt} x_{\tau_k}(t) = \phi^1(t) \mathds{1}_{[-\Delta,0)}(t) + \mathds{1}_{[0,\bar t_f]}(t) \bigg( \phi^1(0) + \int_{0}^{t} f(s,s-\tau^0_k,x_{\tau_k}(s),x_{\tau_k}(s-\tau^1_k),u_{\tau_k}(s),u_{\tau_k}(s-\tau^2_k)) \; ds \bigg) .
\end{equation}
\endgroup
From this, Assumptions $(A_1)$, $(A_4)$ ensure that $(x_{\tau_k}(\cdot))_{k \in \mathbb{N}}$ is bounded in $H^1$, and then, it converges to some $\bar x(\cdot) \in H^1([-\Delta,\bar t_f],\mathbb{R}^n)$ for the weak topology of $H^1$. Since the immersion of $H^1$ into $C^0$ is compact, up to a subsequence, $(x_{\tau_k}(\cdot))_{k \in \mathbb{N}}$ converges to $\bar x(\cdot) \in C^0([-\Delta,\bar t_f],\mathbb{R}^n)$ uniformly in $[-\Delta,\bar t_f]$. Passing to the limit in (\ref{Ch6_ExprDynDelay}) gives
$$
\bar x(t) = \phi^1(t) \mathds{1}_{[-\Delta,0)}(t) + \mathds{1}_{[0,\bar t_f]}(t) \bigg( \phi^1(0) + \int_{0}^{t} f(s,\bar x(s),\bar x(s),\bar u(s),\bar u(s)) \; ds \bigg) .
$$
In particular, one has $\bar x(\bar t_f) \in M_f$, and then, $\bar u(\cdot)$ is admissible for (\textbf{OCP}).

Similarly to the previous case, thanks to the achieved convergences and Assumption $(C_1)$, one proves that $ C_{0}(\bar t_f,\bar u) \leqslant  C_{0}(t_f,u)$. Therefore, from Assumption $(A_2)$, we infer that $\bar t_f = t_f$, $\bar u(\cdot) = u(\cdot)$ and $\bar x(\cdot) = x(\cdot)$ and the conclusion follows.

In this case, not only we have weak convergence of the dynamics and of their derivatives, but also of optimal controls (under appropriate topologies). \\

The convergence almost everywhere of the optimal controls can be achieved when the second option of Assumption $(C_3)$ holds, and more specifically, when $u(\cdot)$ assumes its values at extremal points of $U$, almost everywhere in $[-\Delta,t_f]$.

We proceed as follows. The previous computations provide that $(u_{\tau_k}(\cdot))_{k \in \mathbb{N}}$ converges to $u(\cdot)$ for the weak topology of $L^2$. At this step, the fact that control $u(\cdot)$ assumes its values at extremal points of $U$, almost everywhere in $[-\Delta,t_f]$, implies that $(u_{\tau_k}(\cdot))_{k \in \mathbb{N}}$ converges to $u(\cdot)$ for the strong topology of $L^1$ (see \cite[Corollary 1]{visintin1984}). Therefore, up to some subsequence, $(u_{\tau_k}(\cdot))_{k \in \mathbb{N}}$ converges to $u(\cdot)$, a.e. in $[-\Delta,t_f]$.

\begin{remark} \label{Ch6_RemarkShift}
Up to some subsequence, thanks to the computations in (\ref{Ch6_ShiftExpression}), both $(u_{\tau_k}(\cdot-\tau^2_k))_{k \in \mathbb{N}}$ and $(u_{\tau_k}(\cdot+\tau^2_k))_{k \in \mathbb{N}}$ converges to $u(\cdot)$, almost everywhere in $[-\Delta,t_f]$.
\end{remark}

We have shown that $(t_f,x(\cdot),u(\cdot))$ (substituted by $(t_f,x(\cdot),\dot{x}(\cdot))$ for the case of pure state delays) is the unique closure point (for the topologies used above) of $(t^{\tau_k}_f,x_{\tau_k}(\cdot),u_{\tau_k}(\cdot))_{k \in \mathbb{N}}$ (substituted by $(t^{\tau_k}_f,x_{\tau_k}(\cdot),\dot{x}_{\tau_k}(\cdot))_{k \in \mathbb{N}}$ for the cases of pure state delays), for any (sub)sequence of delays $(\tau_k)_{k \in \mathbb{N}}$ converging to 0. Then, convergence holds as well for the whole family $(t^{\tau}_f,x_{\tau}(\cdot),u_{\tau}(\cdot))_{\tau \in (0,\varepsilon_0)^3}$ (substituted by $(t^{\tau}_f,x_{\tau}(\cdot),\dot{x}_{\tau}(\cdot))_{\tau \in (0,\varepsilon_0)^3}$ for the cases of pure state delays).

\subsubsection{Convergence of Optimal Adjoint Vectors for (\textbf{OCP})$_{\tau}$} \label{Ch6_ConvAdjoint}

In what follows, $(x_{\tau}(\cdot),u_{\tau}(\cdot))$ will denote an optimal solution for (\textbf{OCP})$_{\tau}$ defined in the interval $[-\Delta,t^{\tau}_f]$ such that, if needed, it is extended continuously in $[-\Delta,t_f]$. From the Maximum Principle related to (\textbf{OCP})$_{\tau}$, the trajectory $x_{\tau}(\cdot)$ is the projection of an extremal $(x_{\tau}(\cdot),p_{\tau}(\cdot),p^0_{\tau},u_{\tau}(\cdot))$ which satisfies equations (\ref{Ch5_Adjoint}). From now on, we consider that either Assumptions $(B)$ or Assumptions $(C)$ are satisfied, depending on whether we consider pure state delays or not. The main step of this part consists in showing the convergence of the Pontryagin cone of (\textbf{OCP})$_{\tau}$ to the Pontryagin cone of (\textbf{OCP}). Since the definition of variation vectors relies on Lebesgue points of optimal controls, we need first a set of converging Lebesgue points. Finally, for sake of concision, we do not consider final conditions on the state. Recovering the desired convergence results equipped with transversality conditions can be easily done by traveling back the arguments that follow, and using Assumption $(A_1)$.

\begin{lemma} \label{Ch6_Lemma1}
Consider (\textbf{OCP})$_{\tau}$ with pure state delays and assume that Assumption $(B_1)$ holds. For every $s \in (0,t_f)$, Lebesgue point of function $\tilde f(\cdot,x(\cdot),x(\cdot),u(\cdot))$, there exists a family $(s_{\tau})_{(\tau^0,\tau^1) \in (0,\varepsilon_0)^2} \subseteq [s,t_f)$, which are Lebesgue points of function $\tilde f(\cdot,\cdot-\tau^0,x_{\tau}(\cdot),x_{\tau}(\cdot-\tau^1),u_{\tau}(\cdot))$, such that
$$
\tilde f(s_{\tau},s_{\tau}-\tau^0,x_{\tau}(s_{\tau}),x_{\tau}(s_{\tau}-\tau^1),u_{\tau}(s_{\tau})) \xrightarrow{\tau \rightarrow 0} \tilde f(s,s,x(s),x(s),u(s)) \quad , \quad s_{\tau} \xrightarrow{\tau \rightarrow 0} s .
$$
Conversely, consider (\textbf{OCP})$_{\tau}$ with general delays $\tau = (\tau^0,\tau^1,\tau^2) \in (0,\varepsilon_0)^3$ and assume that Assumption $(C_3)$ holds. For every $s \in (0,t_f)$, Lebesgue point of $u(\cdot)$, there exists a family $(s_{\tau})_{\tau \in (0,\varepsilon_0)^3} \subseteq [s,t_f)$, which are Lebesgue points of $u_{\tau}(\cdot)$, of $u_{\tau}(\cdot-\tau^2)$ and of $u_{\tau}(\cdot+\tau^2)$, such that
$$
u_{\tau}(s_{\tau}) \xrightarrow{\tau \rightarrow 0} u(s) \ , \ u_{\tau}(s_{\tau}-\tau^2) \xrightarrow{\tau \rightarrow 0} u(s) \ , \ u_{\tau}(s_{\tau}+\tau^2) \xrightarrow{\tau \rightarrow 0} u(s) \ , \ s_{\tau} \xrightarrow{\tau \rightarrow 0} s .
$$
\end{lemma}

\begin{proof}[Proof of Lemma \ref{Ch6_Lemma1}] We start by proving the first assertion. For this, denote
$$
h^{\tau}(t) = (h^{\tau}_1(t),\dots,h^{\tau}_{n+1}(t)) = \tilde f(t,t-\tau^0,x_{\tau}(t),x_{\tau}(t-\tau^1),u_{\tau}(t))
$$
$$
h(t) = (h_1(t),\dots,h_{n+1}(t)) = \tilde f(t,t,x(t),x(t),u(t)) .
$$
We prove that, for $s \in (0,t_f)$ Lebesgue point of $h(\cdot)$, for every $\beta > 0$, $\alpha_s > 0$ (such that $s+\alpha_s<t_f$), there exists $\gamma_{s,\alpha_s,\beta} > 0$ such that, for every $(\tau^0,\tau^1) \in (0,\gamma_{s,\alpha_s,\beta})^2$, there exists $s_{\tau} \in [s,s+\alpha_s]$ Lebesgue point of $h^{\tau}(\cdot)$ for which $\| h^{\tau}(s_{\tau}) - h(s) \| < \beta$. By contradiction, suppose that there exists $s \in (0,t_f)$ Lebesgue point of $h(\cdot)$, $\alpha_{s} > 0$, $\beta > 0$ such that, for every integer $k$, there exists $\tau_k \in (0,1/k)^2 \times \{ 0 \}$ and $i_k \in \{ 1,\dots,n+1 \}$ for which, for $t \in [s,s+\alpha_{s}]$ Lebesgue point of $h^{\tau_k}(\cdot)$, it holds $|h^{\tau_k}_{i_k}(t) - h_{i_k}(s)| \geqslant \beta$.

From the previous results, the family $(h^{\tau}(\cdot))_{\tau \in (0,\varepsilon_0)^2 \times \{ 0 \}}$ converges to $h(\cdot)$ in $L^{\infty}$ for the weak star topology. Therefore, for every $0 < \delta \leqslant 1$, there exists an integer $k_{\delta}$ such that, for every $k \geqslant k_{\delta}$, it holds
$$
\frac{1}{\delta \alpha_s} \Big| \int_s^{s+\delta \alpha_s} h^{\tau_k}_i(t) \ \mathrm{d}t - \int_s^{s+\delta \alpha_s} h_i(t) \ \mathrm{d}t \Big| < \frac{\beta}{3} .
$$
for every $i \in \{1,\dots,n+1\}$. We exploit this fact to bound $|h^{\tau_k}_{i_k}(t) - h_{i_k}(s)|$ by $\beta$.

Firstly, since $s$ is a Lebesgue point of $h(\cdot)$, there exists $0 < \delta_{s,\alpha_s} \leqslant 1$ such that
$$
\Big| h_i(s) - \frac{1}{\delta_{s,\alpha_s} \alpha_s} \int_s^{s+\delta_{s,\alpha_s} \alpha_s} h_i(t) \ \mathrm{d}t \Big| < \frac{\beta}{3}
$$
for $i \in \{1,\dots,n+1\}$. On the other hand, there exists an integer $k_{\delta_{s,\alpha_s}}$ such that
$$
\frac{1}{\delta_{s,\alpha_s} \alpha_s} \Big| \int_s^{s+\delta_{s,\alpha_s} \alpha_s} h^{\tau_k}_i(t) \ \mathrm{d}t - \int_s^{s+\delta_{s,\alpha_s} \alpha_s} h_i(t) \ \mathrm{d}t \Big| < \frac{\beta}{3}
$$
for every $k \geqslant k_{\delta_{s,\alpha_s}}$ and every $i \in \{1,\dots,n+1\}$. Finally, by assumption, we have that $h^{\tau}(\cdot)$ is continuous for $\tau^0 , \tau^1 > 0$, and then, for every $k \geqslant k_{\delta_{s,\alpha_s}}$ and every $i \in \{1,\dots,n+1\}$, there exists $t_{k,i} \in [s,s+\delta_{s,\alpha_s} \alpha_s] \subseteq [s,s+\alpha_s]$ such that
$$
\Big| h^{\tau_k}_i(t_{k,i}) - \frac{1}{\delta_{s,\alpha_s} \alpha_s} \int_s^{s+\delta_{s,\alpha_s} \alpha_s} h^{\tau_k}_i(t) \ \mathrm{d}t \Big| < \frac{\beta}{3} .
$$
Then, for every $\tau_k \in \left(0,\frac{1}{k_{\delta_{s,\alpha_s}}}\right)^2 \times \{ 0 \}$, $i \in \{ 1,\dots,n+1 \}$ there exists $t_{k,i} \in [s,s+\alpha_s]$ Lebesgue point of $h^{\tau_k}(\cdot)$ such that $|h^{\tau_k}_i(t_{k,i}) - h_i(s)| < \beta$, which is a contradiction.

Now, we consider the second statement. The case for which Assumption $(C_3)$ ensures that, for every delay $\tau$, every optimal control $u_{\tau}(\cdot)$ of (\textbf{OCP})$_{\tau}$ is continuous, is treated as above because of the weak convergence in $L^2$ of $u_{\tau}(\cdot)$, of $u_{\tau}(\cdot-\tau^2)$ and of $u_{\tau}(\cdot+\tau^2)$. Therefore, assume that $u(\cdot)$ takes its values at extremal points of $U$, almost everywhere in $[-\Delta,t_f]$. Without loss of generality, we extend $u_{\tau}(\cdot)$ by some constant vector of $U$ in $[t^{\tau}_f,b]$. Denote
$$
h^{\tau}(t) = (h^{\tau}_1(t),\dots,h^{\tau}_{3m}(t)) = \Big( u_{\tau}(t),u_{\tau}(t-\tau^2),u_{\tau}(t+\tau^2) \Big)
$$
$$
h(t) = (h_1(t),\dots,h_{3m}(t)) = \Big( u(t),u(t),u(t) \Big)
$$
and fix $s \in (0,t_f)$, Lebesgue point of $h(\cdot)$. By contradiction, suppose that there exist $\beta > 0$ and $\alpha > 0$ such that, for every integer $k$, there exist $\tau_k = (\tau^0_k,\tau^1_k,\tau^2_k) \in (0,1/k)^3$ and $i_k \in \{ 1,\dots,3m \}$ for which, for every $r \in [s,s+\alpha]$ Lebesgue point of $h^{\tau_k}(\cdot)$, it holds $| h^{\tau_k}_{i_k}(r) - h_{i_k}(s) | \geqslant \beta$. From the arguments of the previous sections, up to some extension, the family of controls $(u_{\tau}(\cdot))_{\tau \in (0,\varepsilon_0)^3}$ converges to $u(\cdot)$ almost everywhere in $[0,t_f]$ and the same holds true for $(u_{\tau}(\cdot-\tau^2))_{\tau \in (0,\varepsilon_0)^3}$ and $(u_{\tau}(\cdot+\tau^2))_{\tau \in (0,\varepsilon_0)^3}$, thanks to Remark \ref{Ch6_RemarkShift}. Then, $(h^{\tau_k}_{i}(\cdot))_{k \in \mathbb{N}}$ converges a.e. to $h_i(\cdot)$, raising a contradiction.
\end{proof}

Lemma \ref{Ch6_Lemma1} allows to prove the following property for Pontryagin cones.

\begin{lemma} \label{Ch6_Prop1}
For every $\tilde v \in \tilde K^0(t_f)$ and every $\tau = (\tau^0,\tau^1,\tau^2) \in (0,\varepsilon_0 )^3$ (as well as $\tau = (\tau^0,\tau^1,0) \in (0,\varepsilon_0)^2 \times \{ 0 \}$ in the case of pure state delays), there exists $\tilde w_{\tau} \in \tilde K^{\tau}(t^{\tau}_f)$ such that the family $(\tilde w_{\tau})_{\tau \in (0,\varepsilon_0 )^3}$ converges to $\tilde v$ as $\tau$ tends to 0.
\end{lemma}
\begin{proof}[Proof of Lemma \ref{Ch6_Prop1}] We prove the statement for problems (\textbf{OCP})$_{\tau}$ with general state and control delays $\tau = (\tau^0,\tau^1,\tau^2)$. If pure state delay problems (\textbf{OCP})$_{\tau}$ are considered, the same guideline can be employed by using Lemma \ref{Ch6_Lemma1} and (\ref{Ch6_DerDynConv}).

Suppose first that $\tilde v = \tilde v^0_{s,\omega_z(s)}(t_f)$, where $z \in U$ and $0<s<t_f$ is a Lebesgue point of $u(\cdot)$ (recall Remark \ref{Ch6_RemarkVariations}). By definition, $\tilde v^0_{s,\omega_z(s)}(\cdot)$ is the solution of
\begin{eqnarray} \label{Ch6_SysProp1}
\hspace{10pt} \begin{cases}
\dot{\psi}(t) = \displaystyle \bigg( \frac{\partial \tilde f}{\partial x}(t,t,x(t),x(t),u(t),u(t)) + \frac{\partial \tilde f}{\partial y}(t,t,x(t),x(t),u(t),u(t)) \bigg) \psi(t) \medskip \\
\psi(s) = \tilde f(s,s,x(s),x(s),z,z) - \tilde f(s,s,x(s),x(s),u(s),u(s))
\end{cases} .
\end{eqnarray}
From Lemma \ref{Ch6_Lemma1}, there exists a family $(s_{\tau})_{\tau \in (0,\varepsilon_0)^3} \subseteq [s,t_f)$, which are Lebesgue points of $u_{\tau}(\cdot)$, of $u_{\tau}(\cdot-\tau^2)$ and of $u_{\tau}(\cdot+\tau^2)$, such that
$$
u_{\tau}(s_{\tau}) \xrightarrow{\tau \rightarrow 0} u(s) \ , \ u_{\tau}(s_{\tau}-\tau^2) \xrightarrow{\tau \rightarrow 0} u(s) \ , \ u_{\tau}(s_{\tau}+\tau^2) \xrightarrow{\tau \rightarrow 0} u(s) \ , \ s_{\tau} \xrightarrow{\tau \rightarrow 0} s .
$$
This allows to consider $\tilde v^{\tau}_{s_{\tau},\omega^-_z(s_{\tau})}(\cdot)$ and $\tilde v^{\tau}_{s_{\tau}+\tau^2,\omega^+_z(s_{\tau})}(\cdot)$, solutions of (\ref{Ch6_DynVariation}) with initial data provided respectively by (\ref{Ch6_OmegaMinus}) and (\ref{Ch6_OmegaPlus}). We denote
$$
\tilde w^{\tau}_{s_{\tau},z}(t) = \tilde v^{\tau}_{s_{\tau},\omega^-_z(s_{\tau})}(t) + \tilde v^{\tau}_{s_{\tau}+\tau^2,\omega^+_z(s_{\tau})}(t) .
$$
Since we consider affine problems (\textbf{OCP})$_{\tau}$, Lemma \ref{Ch6_Lemma1} gives
$$
\lim_{\tau \rightarrow 0} \Big( \omega^-_z(s_{\tau}) + \omega^+_z(s_{\tau}) \Big) = \tilde f(s,s,x(s),x(s),z,z) - \tilde f(s,s,x(s),x(s),u(s),u(s)) .
$$
Moreover, from the results of the previous sections, we have in particular
$$
\frac{\partial \tilde f}{\partial x}(\cdot,\cdot-\tau^0,x_{\tau}(\cdot),x_{\tau}(\cdot-\tau^1),u_{\tau}(\cdot),u_{\tau}(\cdot-\tau^2)) \overset{L^2}{\rightharpoonup} \frac{\partial \tilde f}{\partial x}(\cdot,\cdot,x(\cdot),x(\cdot),u(\cdot),u(\cdot))
$$
$$
\frac{\partial \tilde f}{\partial y}(\cdot,\cdot-\tau^0,x_{\tau}(\cdot),x_{\tau}(\cdot-\tau^1),u_{\tau}(\cdot),u_{\tau}(\cdot-\tau^2)) \overset{L^2}{\rightharpoonup} \frac{\partial \tilde f}{\partial y}(\cdot,\cdot,x(\cdot),x(\cdot),u(\cdot),u(\cdot)) .
$$
By continuous dependence w.r.t initial data for dynamical systems and since $t^{\tau}_f$ converges to $t_f$, the family $(\tilde w_{\tau})_{\tau \in (0,\varepsilon_0 )^3} = (\tilde w^{\tau}_{s_{\tau},z}(t^{\tau}_f))_{\tau \in (0,\varepsilon_0 )^3}$ converges to $\tilde v$ as $\tau \rightarrow 0$.

If $\tilde v \in \partial \tilde K^0(t_f)$, the result above is used on converging sequences in $\textnormal{Int} \ \tilde K^0(t_f)$.
\end{proof}

For the last part of the proof, an iterative use of Lemma \ref{Ch6_Prop1} is done. It is at this step that, for problems with general delays $\tau = (\tau^0,\tau^1,\tau^2)$, Assumption $(C_2)$ of fixed final time becomes instrumental to derive the convergence related to adjoint vectors. Indeed, problems arise when one tries to make the final condition on the Hamiltonian (\ref{Ch5_FinalCond}) converge to the transversality condition related to problem (\textbf{OCP})$_{\tau}$. For sake of concision, in this context, we focus only on problems (\textbf{OCP})$_{\tau}$ with general delays $\tau = (\tau^0,\tau^1,\tau^2)$. The case of pure state delays is similar (we refer to \cite[Proposition 2.15]{haberkorn2011} for details). Assumptions $(B)$ and $(C)$ are implicitly used.

We first prove that the following statements hold true:
\begin{itemize}
\item For every $\tau = (\tau^0,\tau^1,\tau^2) \in (0,\varepsilon_0)^3$, every extremal lift $(x_{\tau}(\cdot),p_{\tau}(\cdot),p^0_{\tau},u_{\tau}(\cdot))$ of any solution of (\textbf{OCP})$_{\tau}$ is normal.
\item The set of final adjoint vectors $\{ p_{\tau}(t_f) : \tau \in (0,\varepsilon_0)^3 \}$ is bounded.
\end{itemize}

We consider the first statement proceeding by contradiction. Assume that, for every integer $k$, there exist $\tau_k = (\tau^0_k,\tau^1_k,\tau^2_k) \in (0,1/k)^3$ and a solution $(x_{\tau_k}(\cdot),u_{\tau_k}(\cdot))$ of (\textbf{OCP})$_{\tau_k}$ having an abnormal extremal lift $(x_{\tau_k}(\cdot),p_{\tau_k}(\cdot),0,u_{\tau_k}(\cdot))$. Set $\displaystyle \psi_{\tau_k} = \frac{p_{\tau_k}(t_f)}{\| p_{\tau_k}(t_f) \|}$ for every integer $k$. Therefore, we have $\Big\langle (\psi_{\tau_k},0) , \tilde v_{\tau_k} \Big\rangle \leqslant 0$, for every $\tilde v_{\tau_k} \in \tilde K^{\tau_k}(t_f)$ and every integer $k$. Up to a subsequence, the sequence $(\psi_{\tau_k})_{k \in \mathbb{N}} \subseteq S^{n-1}$ converges to some unit vector $\psi \in \mathbb{R}^n$. Passing to the limit, by using the previous results, we infer that $\Big\langle (\psi,0) , \tilde v \Big\rangle \leqslant 0$ for every $\tilde v \in \tilde K^0(t_f)$. Thanks to Assumption $(C_2)$, $(x(\cdot),u(\cdot))$ has an abnormal extremal lift. This contradicts Assumption $(A_3)$.

For the second statement, again by contradiction, assume that there exists a sequence $(\tau_k = (\tau^0_k,\tau^1_k,\tau^2_k))_{k \in \mathbb{N}} \subseteq (0,\varepsilon_0)^3$ converging to 0 such that $\| p_{\tau_k}(t_f) \|$ tends to $+\infty$. As defined above, the sequence $\displaystyle \left( \psi_{\tau_k} \right)_{k \in \mathbb{N}}$ belongs to $S^{n-1}$, and then, up to some subsequence, it converges to some unit vector $\psi$. On the other hand, by construction, the inequality $\Big\langle (p_{\tau_k}(t_f),-1) , \tilde v_{\tau_k} \Big\rangle \leqslant 0$ holds for every $\tilde v_{\tau_k} \in \tilde K^{\tau_k}(t_f)$ and every integer $k$. Dividing by $\| p_{\tau_k}(t_f) \|$ and passing to the limit, it follows that the solution $(x(\cdot),u(\cdot))$ has an abnormal extremal lift, which again contradicts Assumption $(A_3)$. \\

Now, let $\psi$ be a closure point of $\{ p_{\tau}(t_f) : \tau \in (0,\varepsilon_0)^3 \}$ and $(\tau_k = (\tau^0_k,\tau^1_k,\tau^2_k))_{k \in \mathbb{N}} \subseteq (0,\varepsilon_0)^3$ be a sequence converging to 0 such that $p_{\tau_k}(t_f)$ tends to $\psi$.  Using the continuous dependence w.r.t. initial data and the established convergence properties, we infer that $(p_{\tau_k}(\cdot))_{k \in \mathbb{N}}$ converges uniformly to the solution $z(\cdot)$ of
\begingroup
\small
$$
\displaystyle \dot{z}(t) = -\frac{\partial H}{\partial x}(t,t,x(t),x(t),z(t),-1,u(t),u(t)) - \frac{\partial H}{\partial y}(t,t,x(t),x(t),z(t),-1,u(t),u(t))
$$
\endgroup
 with $z(t_f) = \psi$. Moreover, since $\Big\langle (p_{\tau_k}(t_f),-1) , \tilde v_{\tau_k} \Big\rangle \leqslant 0$, for every $\tilde v_{\tau_k} \in \tilde K^{\tau_k}(t_f)$ and every integer $k$, passing to the limit, thanks to the previous results, we obtain $\Big\langle (\psi,-1) , \tilde v \Big\rangle \leqslant 0$, for every $\tilde v \in \tilde K^0(t_f)$. It follows that $(x(\cdot),z(\cdot),-1,u(\cdot))$ is a normal extremal lift of (\textbf{OCP}). Using Assumption $(A_3)$, we finally obtain $z(\cdot) = p(\cdot)$ in $[0,t_f]$.

Theorem \ref{Ch5_TheoMain} is proved.

\section{Conclusions and Perspectives} \label{Sect_Conclus}

In this paper, we provided sufficient conditions under which Pontryagin extremals related to nonlinear optimal control problems with delays are continuous (under appropriate topologies) with respect to delays.

It would be interesting to extend this result to problems with more general constraints,
such as, control and state constraints. This would require to analyze the proof of the Maximum Principle with state and control constraints via sliding or v-variations (see, e.g., \cite{dmitruk2009}), to exploit the continuous dependence with respect to parameters for implicit function theorems (Ekeland-type approaches probably fail because of continuous dependence). Furthermore, in the case of control and state delays, the proof that we provided needs to consider control-affine dynamics and costs, and fixed final time. The extension to more general systems is open.

Furthermore, in this paper we have considered constant delays. One could consider more general delays that are functions of time and state. This is motivated by the fact that, in the case of delays depending on the time and the state, Maximum Principle formulations still exist (see, e.g., \cite{asher1971}). Therefore, extending our main result requires to consider the $C^0$-topology on the delay function $t \mapsto \tau(t,x(t))$.

Finally, as shortly explained in Section \ref{Sect_IndirectMethods}, our result is in particular motivated by numerical implementations of the shooting method, in combination with homotopies on the delay parameters, thus providing an interesting alternative to classically used direct methods. Numerical issues will be addressed in a forthcoming paper.

\bibliographystyle{unsrt}
\bibliography{references}

\appendix

\section{Proof of of Lemma \ref{Ch6_LemmaNeedleLike}} \label{appendixProof}

The proof goes by induction. We develop computations for $j = 1$. The inductive step goes in the same way, as the usual case (see, e.g., \cite{pontryagin1987}).

Let $t_j < t \leqslant t^{\tau}_f$ be a Lebesgue point of $u_{\tau}(\cdot)$, $u_{\tau}(\cdot-\tau^2)$. First, let us show that
\begingroup
\small
\begin{equation} \label{Ch6_ExprProofNeedle}
\tilde x^{\pi}_{\tau}(t) - \tilde x_{\tau}(t) = \eta_1 \tilde w^{\tau}_{t_1,u_1}(t) + o (\eta_1) = \eta_1 \Big( \tilde v^{\tau}_{t_1,\omega^-_{u_1}(t_1)}(t) + \tilde v^{\tau}_{t_1+\tau^2,\omega^+_{u_1}(t_1)}(t) \Big) + o (\eta_1) .
\end{equation}
\endgroup
We consider the case $t \geqslant t_1 + \tau_2$ (the case $t < t_1 + \tau_2$ is similar, but easier). We have
\begingroup
\footnotesize
$$
\hspace{-50pt} \| \tilde x^{\pi}_{\tau}(t) - \tilde x_{\tau}(t) - \eta_1 \tilde w^{\tau}_{t_1,u_1}(t) \| \leqslant \| \tilde x^{\pi}_{\tau}(t_1+\tau^2) - \tilde x_{\tau}(t_1+\tau^2) - \eta_1 \tilde w^{\tau}_{t_1,u_1}(t_1+\tau^2) \|
$$
$$
\hspace{-120pt} + \bigg\| \int^{t}_{t_1+\tau_2} \Big( \tilde f(s,s-\tau^0,\tilde x^{\pi}_{\tau}(s),\tilde x^{\pi}_{\tau}(s-\tau^1),u_{\tau}(s),u_{\tau}(s-\tau^2))
$$
$$
\hspace{85pt} - \tilde f(s,s-\tau^0,\tilde x_{\tau}(s),\tilde x_{\tau}(s-\tau^1),u_{\tau}(s),u_{\tau}(s-\tau^2)) - \eta_1 \dot{\tilde{w}}^{\tau}_{t_1,u_1}(s) \Big) ds \bigg\| \ .
$$
\endgroup
By exploiting the facts that $t_1$ is a Lebesgue point of $u_{\tau}(\cdot)$, $u_{\tau}(\cdot-\tau^2)$ and of $u_{\tau}(\cdot+\tau^2)$, and that $\tilde x^{\pi}_{\tau}(\cdot)$ converges uniformly to $\tilde x_{\tau}(\cdot)$, expanding the extended dynamics at second order, the first term of the expression above can be bounded as follows:
\begingroup
\footnotesize
$$
\hspace{-190pt} \| \tilde x^{\pi}_{\tau}(t_1+\tau^2) - \tilde x_{\tau}(t_1+\tau^2) - \eta_1 \tilde w^{\tau}_{t_1,u_1}(t_1+\tau^2) \|
$$
$$
\hspace{-150pt} \leqslant \bigg\| \int^{t_1}_{t_1-\eta_1} \Big( \tilde f(s,s-\tau^0,\tilde x^{\pi}_{\tau}(s),\tilde x^{\pi}_{\tau}(s-\tau^1),u_1,u_{\tau}(s-\tau^2))
$$
$$
\hspace{125pt} - \tilde f(s,s-\tau^0,\tilde x_{\tau}(s),\tilde x_{\tau}(s-\tau^1),u_{\tau}(s),u_{\tau}(s-\tau^2)) \Big) ds - \eta_1 \omega^-_{u_1}(t_1) \bigg\|
$$
$$
\hspace{-107pt} + \bigg\| \int^{t_1+\tau^2-\eta_1}_{t_1} \Big( \tilde f(s,s-\tau^0,\tilde x^{\pi}_{\tau}(s),\tilde x^{\pi}_{\tau}(s-\tau^1),u_{\tau}(s),u_{\tau}(s-\tau^2))
$$
$$
\hspace{125pt} - \tilde f(s,s-\tau^0,\tilde x_{\tau}(s),\tilde x_{\tau}(s-\tau^1),u_{\tau}(s),u_{\tau}(s-\tau^2)) - \eta_1 \dot{\tilde{w}}^{\tau}_{t_1,u_1}(s) \Big) ds \bigg\|
$$
$$
\hspace{-140pt} + \bigg\| \int^{t_1+\tau^2}_{t_1+\tau^2-\eta_1} \Big( \tilde f(s,s-\tau^0,\tilde x^{\pi}_{\tau}(s),\tilde x^{\pi}_{\tau}(s-\tau^1),u_{\tau}(s),u_1)
$$
$$
\hspace{85pt} - \tilde f(s,s-\tau^0,\tilde x_{\tau}(s),\tilde x_{\tau}(s-\tau^1),u_{\tau}(s),u_{\tau}(s-\tau^2)) - \eta_1 \dot{\tilde{w}}^{\tau}_{t_1,u_1}(s) \Big) ds - \eta_1 \omega^+_{u_1}(t_1) \bigg\|
$$
\endgroup
\begingroup
\footnotesize
$$
\hspace{-150pt} \leqslant \bigg\| \int^{t_1}_{t_1-\eta_1} \Big( \tilde f(s,s-\tau^0,\tilde x_{\tau}(s),\tilde x_{\tau}(s-\tau^1),u_1,u_{\tau}(s-\tau^2))
$$
$$
\hspace{100pt} - \tilde f(s,s-\tau^0,\tilde x_{\tau}(s),\tilde x_{\tau}(s-\tau^1),u_{\tau}(s),u_{\tau}(s-\tau^2)) \Big) ds - \eta_1 \omega^-_{u_1}(t_1) \bigg\| + o (\eta_1)
$$
$$
+ \int^{t_1+\tau^2-\eta_1}_{t_1} \bigg\| \frac{\partial \tilde{f}}{\partial x}(s,s-\tau^0,\tilde x_{\tau}(s),\tilde x_{\tau}(s-\tau^1),u_{\tau}(s),u_{\tau}(s-\tau^2)) \cdot \Big(\tilde x^{\pi}_{\tau}(s) - \tilde x_{\tau}(s) - \eta_1 \tilde w^{\tau}_{t_1,u_1}(s)\Big) \bigg\| ds
$$
$$
+ \int^{t_1+\tau^2-\eta_1}_{t_1} \bigg\| \frac{\partial \tilde{f}}{\partial y}(s,s-\tau^0,\tilde x_{\tau}(s),\tilde x_{\tau}(s-\tau^1),u_{\tau}(s),u_{\tau}(s-\tau^2)) \cdot \Big(\tilde x^{\pi}_{\tau} - \tilde x_{\tau} - \eta_1 \tilde w^{\tau}_{t_1,u_1}\Big)(s-\tau^1) \bigg\| ds
$$
$$
+ \int^{t_1+\tau^2-\eta_1}_{t_1} \int_{0}^{1} \bigg\| d^2 \tilde f\big(s,s-\tau^0,(\sigma \tilde x_{\tau} + (1-\sigma ) \tilde x^{\pi}_{\tau})(s),(\sigma \tilde x_{\tau} + (1-\sigma ) \tilde x^{\pi}_{\tau})(s-\tau^1),u_{\tau}(s),u_{\tau}(s-\tau^2)\big) \bigg\| \cdot
$$
$$
\cdot \bigg( \| \tilde x^{\pi}_{\tau}(s) - \tilde x_{\tau}(s) \|^2 + \| \tilde x^{\pi}_{\tau}(s-\tau^1) - \tilde x_{\tau}(s-\tau^1) \|^2 + 2 \| \tilde x^{\pi}_{\tau}(s) - \tilde x_{\tau}(s) \| \| \tilde x^{\pi}_{\tau}(s-\tau^1) - \tilde x_{\tau}(s-\tau^1) \| \bigg)  d\sigma ds
$$
$$
\hspace{-5pt} + \eta_1 \bigg\| \int^{t_1}_{t_1-\eta_1} \dot{\tilde{w}}^{\tau}_{t_1,u_1}(s+\tau^2) ds \bigg\| + \bigg\| \int^{t_1}_{t_1-\eta_1} \Big( \tilde f(s+\tau^2,s+\tau^2-\tau^0,\tilde x_{\tau}(s+\tau^2),\tilde x_{\tau}(s+\tau^2-\tau^1),u_{\tau}(s+\tau^2),u_1)
$$
$$
\hspace{45pt} - \tilde f(s+\tau^2,s+\tau^2-\tau^0,\tilde x_{\tau}(s+\tau^2),\tilde x_{\tau}(s+\tau^2-\tau^1),u_{\tau}(s+\tau^2),u_{\tau}(s)) \Big) ds - \eta_1 \omega^+_{u_1}(t_1) \bigg\| + o (\eta_1) .
$$
\endgroup
Therefore, by bounding the derivatives of the extended dynamics, we have
\begingroup
\footnotesize
$$
\| \tilde x^{\pi}_{\tau}(t_1+\tau^2) - \tilde x_{\tau}(t_1+\tau^2) - \eta_1 \tilde w^{\tau}_{t_1,u_1}(t_1+\tau^2) \| \leqslant \tilde C_1 \int^{t_1+\tau^2-\eta_1}_{t_1-\tau^1} \| \tilde x^{\pi}_{\tau}(s) - \tilde x_{\tau}(s) - \eta_1 \tilde w^{\tau}_{t_1,u_1}(s) \| \; ds + o (\eta_1)
$$
\endgroup
where $\tilde C_1 \geqslant 0$ is a constant. With the same technique, we obtain
\begingroup
\footnotesize
$$
\hspace{-5pt} \bigg\| \int^{t}_{t_1+\tau_2} \Big( \tilde f(s,s-\tau^0,\tilde x^{\pi}_{\tau}(s),\tilde x^{\pi}_{\tau}(s-\tau^1),u_{\tau}(s),u_{\tau}(s-\tau^2)) - \tilde f(s,s-\tau^0,\tilde x_{\tau}(s),\tilde x_{\tau}(s-\tau^1),u_{\tau}(s),u_{\tau}(s-\tau^2))
$$
\endgroup
\begingroup
\small
$$
- \eta_1 \dot{\tilde{w}}^{\tau}_{t_1,u_1}(s) \Big) ds \bigg\| \leqslant \tilde C_2 \int^{t}_{t_1+\tau^2-\tau^1} \| \tilde x^{\pi}_{\tau}(s) - \tilde x_{\tau}(s) - \eta_1 \tilde w^{\tau}_{t_1,u_1}(s) \| \; ds + o (\eta_1)
$$
\endgroup
where $\tilde C_2 \geqslant 0$ is another constant. Coupling the two last results with the Gr\"onwall's inequality, (\ref{Ch6_ExprProofNeedle}) follows. The conclusion comes from (\ref{Ch6_ExprProofNeedle}) and the fact that $t$ is a Lebesgue point of $u_{\tau}(\cdot)$ and of $u_{\tau}(\cdot-\tau^2)$.

\section{Proof of Lemma \ref{Ch6_IFTP}}  \label{appendixProofSecond}

We start by recalling the following standard result (see, e.g., \cite{agrachev2013}).

Let $L : \mathbb{R}^j \rightarrow \mathbb{R}^n$ be a linear mapping such that $L(\mathbb{R}^j_+) = \mathbb{R}^n$. Then:
\begin{itemize}
\item $j > n + 1$ and $(0,+\infty)^j \cap \textnormal{ker } L$ is nontrivial.
\item There exists $S \subseteq{\mathbb{R}^j}$, dim$(S) = n$, such that $L|_S : S \rightarrow \mathbb{R}^n$ is an isomorphism.
\end{itemize}

Applying this result to $L = \frac{\partial F}{\partial x}(0,0)$ yields the existence of a nontrivial vector $v \in (0, +\infty)^j$, such that $L(v) = 0$, and of a $n$-dimensional subspace $S \subseteq \mathbb{R}^j$ such that the restriction $L|_S : S \rightarrow \mathbb{R}^n$ is an isomorphism.

For every $\varepsilon \in \mathbb{R}^k_+$ and every $y, u \in \mathbb{R}^n$, set $\Phi(\varepsilon,y,u) = u - F(\varepsilon,{L|_S}^{-1}(u)) + y$. This mapping is continuous and it holds $\Phi(0,0,0) = 0$. Fix $\varepsilon \in \mathbb{R}^k_+$ at which $F$ is almost everywhere strictly differentiable. Then, for every $y \in \mathbb{R}^n$, one has
\begingroup
\small
\begin{equation} \label{Ch6_TaylorDev}
\Phi(\varepsilon,y,u_1) - \Phi(\varepsilon,y,u_2) = \left( \textnormal{Id} - \frac{\partial F}{\partial x}(\varepsilon,0) \circ {L|_S}^{-1} \right) (u_1 - u_2) + \| u_1 - u_2 \| G_{\varepsilon}(u_1, u_2)
\end{equation}
\endgroup
where $G_{\varepsilon}(u_1,u_2) = g_{\varepsilon}({L|_S}^{-1}(u_2),{L|_S}^{-1}(u_1)) \rightarrow 0$ as soon as $(u_1,u_2) \xrightarrow{\text{a.e.}} 0$. From the continuity property of $\frac{\partial F}{\partial x}(\varepsilon,0)$ on a dense subset, there exists $\varepsilon_0 \in \mathbb{R}^k_+$ and a dense subset $E \subseteq [0,\varepsilon_0)^k$, such that for every $\varepsilon \in E$
$$
\left\| \textnormal{Id} - \frac{\partial F}{\partial x}(\varepsilon,0) \circ {L|_S}^{-1} \right\| \leqslant \frac{1}{4}
$$
and there exists $r_{\varepsilon} > 0$ such that
$$
\| G_{\varepsilon}(u_1, u_2) \| \leqslant \frac{1}{4} \quad \textnormal{for almost all} \quad u_1 \ , \ u_2 \in B_{r_{\varepsilon}}(0) .
$$
On the other hand, by assumption, the remainder in expression (\ref{Ch6_TaylorDev}) converges to 0 uniformly with respect to $\varepsilon$ on a dense subset. Therefore, up to reducing $E$, gathering the previous results with (\ref{Ch6_TaylorDev}), we infer the existence of $r > 0$ such that
\begingroup
\small
$$
\| \Phi(\varepsilon,y,u_1) - \Phi(\varepsilon,y,u_2) \| \leqslant \frac{1}{2} \| u_1 - u_2 \| \ , \ \textnormal{for every} \ \varepsilon \in E \ \textnormal{and almost every} \ u_1 , u_2 \in B_r(0) .
$$
\endgroup
From this last result and the continuity of mapping $F$, for every $\varepsilon \in [0,\varepsilon_0)^k$ and $y \in \mathbb{R}^n$, the mapping $u \mapsto \Phi(\varepsilon,y,u)$ is $\frac{1}{2}$-Lipschitzian on an open neighborhood of 0.

At this step, for every $\delta > 0$, denote $B_{\delta} = S \cap \overline B_{\delta}(0)$ and choose $\delta > 0$ small enough such that $v + B_{\delta} \subseteq (0,+\infty)^j$. The set $U_{\delta} = L(B_{\delta})$ is a closed neighborhood of 0 in $\mathbb{R}^n$.  With the same argument as above, it is not difficult to show that, if $\delta$, $\| \varepsilon \|$ and $\| y \|$ are small enough, then, the mapping $u \mapsto \Phi(\varepsilon,y,u)$ maps $U_{\delta}$ into itself.

Lemma \ref{Ch6_IFTP} follows from the application of the usual Banach fixed point theorem to the contraction mapping $u \mapsto \Phi(\varepsilon,y,u)$ with parameters $(\varepsilon,y)$.
\end{document}